\documentclass[a4paper,10pt]{scrartcl}

\usepackage[utf8]{inputenc}
\usepackage{amsmath,amssymb,amsthm}
\usepackage{color}
\usepackage{pst-all}
\usepackage{enumerate}
\usepackage[active]{srcltx}
\usepackage{cite}

\newtheorem{Theorem}{Theorem}
\newtheorem{Proposition}{Proposition}
\newtheorem{Lemma}{Lemma}

\newtheoremstyle{Remark}
  {}                   
  {}                   
  {\normalfont}           
  {}                      
  {\normalfont\bfseries}  
  {.}                     
  { }              
  {}
\theoremstyle{Remark}
\newtheorem*{Remark}{Remark}
    
\newenvironment{Beweis}[1][1]{\vspace{1ex}\noindent{\bf Proof of #1:}}
	{\hfill\qed\vspace{2ex}}
	
\DeclareMathOperator{\sfA}{\mathsf{A}}
\DeclareMathOperator{\sfAA}{\sfA\!\sfA}
\DeclareMathOperator{\sfB}{\mathsf{B}}
\DeclareMathOperator{\sfAB}{\sfA\!\sfB}
\DeclareMathOperator{\sfBB}{\sfB\!\sfB}
\DeclareMathOperator{\sft}{\mathsf{t}}
\DeclareMathOperator{\sfs}{\mathsf{s}}
\DeclareMathOperator{\sfx}{\mathsf{x}}
\DeclareMathOperator{\sfy}{\mathsf{y}}
\DeclareMathOperator{\Ext}{\mathsf{Ext}}
\DeclareMathOperator{\Surv}{\mathsf{Surv}}
\DeclareMathOperator{\G}{\mathbb{G}}
\DeclareMathOperator{\N}{\mathbb{N}}
\def\Prob{\mathbb{P}}
\def\E{\mathbb{E}}
\def\Var{\mathbb{V}ar}
\DeclareMathOperator{\T}{\mathbb{T}}

\def\calY{\mathcal{Y}}
\def\calZ{\mathcal{Z}}
\DeclareMathOperator{\Z}{\mathcal{Z}}
\DeclareMathOperator{\1}{\mathbf{1}}
\def\Pstlim{\mathop{\Prob^{*}\text{\rm -lim}\,}}
\def\eps{\varepsilon}

\allowdisplaybreaks[4] 

\title{A host-parasite model for a two-type cell population}
\author{\textsc{Gerold Alsmeyer} and \textsc{S\"oren Gr\"ottrup}\\ \\
\large
        Institut f\"ur Mathematische Statistik,
        Universit\"at M\"unster,\\
\large
        Einsteinstra\ss e 62,
        DE-48149 M\"unster,
        Germany}

\begin{document}

\maketitle

\begin{abstract}
We consider a host-parasite model for a population of cells that can be of two types, $\sfA$ or $\sfB$, and exhibits unilateral reproduction: while a $\sfB$-cell always splits into two cells of the same type, the two daughter cells of an $\sfA$-cell can be of any type. The random mechanism that describes how parasites within a cell multiply and are then shared into the daughter cells is allowed to depend on the hosting mother cell as well as its daughter cells. Focusing on the subpopulation of $\sfA$-cells and its parasites, our model differs from the single-type model recently studied by \textsc{Bansaye} \cite{Bansaye:08} in that the sharing mechanism may be biased towards one of the two types. Our main results are concerned with the nonextinctive case and provide information on the behavior, as $n\to\infty$, of the number $\sfA$-parasites in generation $n$ and the relative proportion of $\sfA$- and $\sfB$-cells in this generation which host a given number of parasites. As in \cite{Bansaye:08}, proofs will 
make use of a so-called random cell line which, when conditioned to be of type $\sfA$, behaves like a branching process in random environment.
\end{abstract}

\section{Introduction}\label{Section.Introduction}

The reciprocal adaptive genetic change of two antagonists (e.g.\ different species or genes)
through reciprocal selective pressures is known as host-parasite coevolution. It may be  observed even in real-time under both, field and laboratory conditions, if reciprocal adaptations take place rapidly and generation times are short. For more information see e.g.\ \cite{Laine:09, Woolhouse:02}.

The present work studies a host-parasite branching model with two types of cells (the hosts), here called $\sfA$ and $\sfB$, and proliferating parasites colonizing the cells. Adopting a genealogical perspective, we are interested in the evolution of certain characteristics over generations and under the following assumptions on the reproductive behavior of cells and parasites. All cells behave independently and split into two daughter cells after one unit of time. The types of the daughter cells of a type-$\sfA$ cell are chosen in accordance with a random mechanism which is the same for all mother cells of this type whereas both daughter cells of a type-$\sfB$ cell are again of type $\sfB$. Parasites within a cell multiply in an iid manner to produce a random number of offspring the distribution of which may depend on the type of this cell as well as on those of its daughter cells. The same holds true for the random mechanism by which the offspring is shared into these daughter cells.

\vspace{.2cm}
The described model grew out of a discussion with biologists in an attempt to provide a first very simple setup that allows to study coevolutionary adaptations, here due to the presence of two different cell types. It may also be viewed as a simple multi-type extension of a model studied by \textsc{Bansaye} \cite{Bansaye:08} which in turn forms a discrete-time version of a model introduced by \textsc{Kimmel} \cite{Kimmel:97}. Bansaye himself extended his results in \cite{Bansaye:09} by allowing immigration and random environments, the latter meaning that each cell chooses the reproduction law for the parasites it hosts in an iid manner. 
Let us further mention related recent work by \textsc{Guyon} \cite{Guyon:07} who studied another discrete-time model with asymmetric sharing and obtained limit theorems under ergodic hypotheses which, however, exclude an extinction-explosion principle for the parasites which is valid in our model.

\vspace{.2cm}
We continue with the introduction of some necessary notation which is similar to the one in \cite{Bansaye:08}. Making the usual assumption of starting from one ancestor cell, denoted as $\varnothing$, we put $\G_{0}:=\{\varnothing\}$, $\G_{n}:=\{0,1\}^{n}$ for $n\ge 1$, and let
\begin{equation*}
	\T := \bigcup_{n\in\mathbb N_0} \G_n\quad\text{with}\quad\G_n := \{0,1\}^n
\end{equation*}
be the binary Ulam-Harris tree rooted at $\varnothing$ which provides the label set of all cells in the considered population. Plainly, $\G_{n}$ contains the labels of all cells of generation $n$. For any cell $v\in\T$, let $T_v\in\{\sfA,\sfB\}$ denote its type and $Z_v$ the number of parasites it contains. \emph{Unless stated otherwise, the ancestor cell is assumed to be of type $\sfA$ and to contain one parasite, i.e.}
\begin{equation}\label{SA1}\tag{SA1}
T_{\varnothing}=\sfA\quad\text{and}\quad Z_{\varnothing}=1.
\end{equation}
Then, for $\sft\in\{\sfA,\sfB\}$ and $n\ge 0$, define
\begin{equation*}
 	\G_n(\sft) := \{v\in\G_n :  T_v=\sft\}\quad\text{and}\quad\G^*_n(\sft) := \{v\in\G_n	(\sft) :  Z_v>0\} 
\end{equation*}
as the sets of type-$\sft$ cells and type-$\sft$ contaminated cells in generation $n$, respectively. The set of all contaminated cells in generation $n$ are denoted $\G^*_n$,
thus $\G^*_n = \G^*_n(\sfA)\cup\G^*_n(\sfB)$. 

As common, we write $v_{1}...v_{n}$ for $v=(v_1,...,v_n)\in\G_n$, $uv$ for the concatenation of $u,v\in\T$, i.e.
\begin{equation*}
	uv=u_1...,u_m v_1...v_n\ \ \text{if}\ u=u_1...u_m\ \text{and}\ v=v_1...v_n,
\end{equation*}
and $v|k$ for the ancestor of $v=v_{1}...v_{n}$ in generation $k\le n$, thus $v|k=v_1,...,v_k$. Finally, if $v|k=u$ for some $k$ and $u\ne v$, we write $u<v$.

\vspace{.2cm}
The process $(T_v)_{v\in\T}$ is a Markov process indexed by the tree $\T$ as defined in \cite{BenPeres:94}. It has transition probabilities
\begin{align*}
	&\Prob(T_{v0}=\sfx,T_{v1}=\sfy|T_v = \sfA) = p_{\sfx\sfy},
	\quad (\sfx,\sfy)\in\{(\sfA,\sfA),(\sfA,\sfB),(\sfB,\sfB)\},\\[1ex]
	&\Prob(T_{v0}=\sfB,T_{v1}=\sfB|T_v = \sfB) = 1,
\end{align*}
and we denote by
\begin{equation*}
 	p_0 := p_{\sfAA} + p_{\sfAB}=1-p_{\sfBB}\quad\text{and}\quad p_1 := p_{\sfAA}
\end{equation*}
the probabilities that the first and the second daughter cell are of type $\sfA$, respectively.
In order to rule out total segregation of type-$\sfA$ and type-$\sfB$ cells, which would just lead back to the model studied in \cite{Bansaye:08}, it will be assumed throughout that
\begin{equation}\label{SA2}\tag{SA2}
p_{\sfAA}<1.
\end{equation}
 
The sequence $(\#\G_n(\sfA))_{n\ge 0}$ obviously forming a Galton-Watson branching process with one ancestor (as $T_{\varnothing}=\sfA$) and mean 
\begin{equation*}
  \nu := p_{0}+p_{1}=2 p_{\sfAA}+p_{\sfAB}=1+(p_{\sfAA}-p_{\sfBB})<2, 
\end{equation*}
it is a standard fact that (see e.g.\ \cite{Athreya+Ney:72})
\begin{equation*}
	\#\G_n(\sfA)\rightarrow 0\text{ a.s.}\quad\text{iff}\quad p_{\sfAA}\leq p_{\sfBB}\quad\text{and}\quad p_{\sfAB}<1.
\end{equation*}

To describe the multiplication of parasites, let $Z_{v}$ denote the number of parasites in cell $v$ and, for $\sft\in\{\sfA,\sfB\}$, $\sfs\in\{\sfAA,\sfAB,\sfBB\}$, let
$$ \left(X^{(0)}_{k,v}(\sft, \sfs),X^{(1)}_{k,v}(\sft, \sfs)\right)_{k\in\mathbb N, v\in\T},\quad\sft\in\{\sfA,\sfB\},\ \sfs\in\{\sfAA,\sfAB,\sfBB\}$$ 
be independent families of iid $\mathbb N^2_0$-valued random vectors with respective generic copies $(X^{(0)}(\sft, \sfs),X^{(1)}(\sft, \sfs))$. If $v$ is of type $\sft$ and their daughter cells are of type $\sfx$ and $\sfy$, then $X^{(i)}_{k,v}(\sft, \sfx\!\sfy)$ gives the offspring number of the $k^{\rm th}$ parasite in cell $v$ that is shared into the daughter cell $vi$ of $v$. Since type-$\sfB$ cells can only produce daughter cells of the same type, we will write $(X^{(0)}_{k,v}(\sfB),X^{(1)}_{k,v}(\sfB))$ as shorthand for $(X^{(0)}_{k,v}(\sfB, \sfBB),X^{(1)}_{k,v}(\sfB, \sfBB))$. To avoid trivialities, it is always assumed hereafter that
\begin{equation}\label{SA3}\tag{SA3}
  \Prob\left(X^{(0)}(\sfA, \sfAA) \leq 1,\ X^{(1)}(\sfA, \sfAA)\leq1\right)<1
\end{equation}
and
\begin{equation}\label{SA4}\tag{SA4}
  \Prob\left(X^{(0)}(\sfB)\leq1,\ X^{(1)}(\sfB)\leq1\right)<1.
\end{equation}

Next, observe that
\begin{equation*}
 	(Z_{v0}, Z_{v1}) = \sum_{\sft\in\{\sfA,\sfB\}}\1_{\{T_{v}=\sft\}}\sum_{\sfs\in\{\sfAA,\sfAB,\sfBB\}}\1_{\{(T_{v0},T_{v1})=\sfs\}}\sum_{k=1}^{Z_v}(X^{(0)}_{k,v}(\sft, \sfs),X^{(1)}_{k,v}(\sft, \sfs)).
\end{equation*}
We put $\mu_{i,\sft}(\sfs) := \E X^{(i)}(\sft, \sfs)$ for $i\in\{0,1\}$ and $\sft,\sfs$ as before, write $\mu_{i,\sfB}$ as shorthand for $\mu_{i,\sfB}(\sfBB)$ and assume throughout that $\mu_{i,\sft}(\sfs)$ are finite and
\begin{equation}\label{SA5}\tag{SA5}
 	\mu_{0,\sfA}(\sfAA),\ \mu_{1,\sfA}(\sfAA),\ \E\left(\#\G^*_1(\sfB)\right)>0,\ \mu_{0,\sfB},\ \mu_{1,\sfB}\ >\ 0.
\end{equation}

The total number of parasites in cells of type $\sft\in\{\sfA, \sfB\}$ at generation $n$ is denoted by
\begin{equation*}
 	\calZ_n(\sft) := \sum_{v\in\G_n(\sft)} Z_v,
\end{equation*}
and we put $\calZ_n := \calZ_n(\sfA)+\calZ_n(\sfB)$, plainly the total number of all parasites at generation $n$. Both, $(\calZ_n)_{n\ge 0}$ and $(\calZ_n(\sfA))_{n\ge 0}$, are transient Markov chains with absorbing state $0$ and satisfy the extinction-explosion principle (see Section I.5 in \cite{Athreya+Ney:72} for a standard argument), i.e.
\begin{equation*}
 	\Prob(\calZ_n\rightarrow 0)+\Prob(\calZ_n\rightarrow\infty)=1\quad\text{and}\quad \Prob(\calZ_n(\sfA)\rightarrow 0)+\Prob(\calZ_n(\sfA)\rightarrow\infty)=1.
\end{equation*}
The extinction events are defined as
\begin{equation*}
 	\Ext := \{\calZ_n\rightarrow 0\}\quad\text{and}\quad\Ext(\sft):=\{\calZ_n(\sft)\rightarrow 0\},\quad\sft\in\{\sfA,\sfB\},
\end{equation*}
their complements by $\Surv$ and $\Surv(\sft)$, respectively.

\vspace{.2cm}
As in \cite{Bansaye:08}, we are interested in the statistical properties of an infinite \emph{random cell line}, picked however from those lines consisting of $\sfA$-cells only. 
This leads to a so-called \emph{random $\sfA$-cell line}. Since $\sfB$-cells produce only daughter cells of the same type, the properties of a random $\sfB$-cell line may be deduced from the afore-mentioned work and are therefore not studied hereafter.

For the definition of a random $\sfA$-cell line, a little more care than in \cite{Bansaye:08} is needed because cells occur in two types and parasitic reproduction may depend on the types of the host and both its daughter cells. On the other hand, we will show in Section \ref{Section.Preliminaries} that a random $\sfA$-cell line still behaves like a branching process in iid random environment (BPRE) which has been a fundamental observation in \cite{Bansaye:08} for a random cell line in the single-type situation.

Let $U=(U_n)_{n\in\mathbb N}$ be an iid sequence of symmetric Bernoulli variables independent of the parasitic evolution and put $V_{n}:=U_{1}...U_{n}$. Then
$$ \varnothing=:V_{0}\to V_{1}\to V_{2}\to...\to V_{n}\to... $$
provides us with a random cell line in the binary Ulam-Haris tree, and we denote by
\begin{equation*}
 	T_{[n]} = T_{V_n}\quad\text{and}\quad Z_{[n]} = Z_{V_n}\quad n\ge 0,
\end{equation*}
the cell types and the number of parasites along that random cell line. A random $\sfA$-cell line up to generation $n$ is obtained when $T_{[n]}=\sfA$, for then $T_{[k]}=\sfA$ for any
$k=0,...,n-1$ as well. As will be shown in Prop.\ \ref{prop:random cell line=BPRE}, the conditional law of $(Z_{[0]},...,Z_{[n]})$ given $T_{[n]}=\sfA$, i.e., given an $\sfA$-cell line up to generation $n$ is picked at random, equals the law of a certain BPRE $(Z_{k}(\sfA))_{k\ge 0}$ up to generation $n$, for each $n\in\N$. It should be clear that this cannot be generally true for the unconditional law of $(Z_{[0]},...,Z_{[n]})$, due to the multi-type structure of the cell population.

\vspace{.2cm}
Aiming at a study of host-parasite coevolution in the framework of a multitype host population, our model may be viewed as the simplest possible alternative. There are only two types of host cells and reproduction is unilateral in the sense that cells of type $\sfA$ may give birth to both, $\sfA$- and $\sfB$-cells, but those of type $\sfB$ will never produce cells of the opposite type. 
The basic idea behind this restriction is that of irreversible mutations that generate new types of cells but never lead back to already existing ones. Observe that our setup could readily be generalized without changing much the mathematical structure by allowing the occurrence of further irreversible mutations from cells of type $\sfB$ to cells of type $\mathsf{C}$, and so on.

\vspace{.2cm}
The rest of this paper is organized as follows. We focus on the case of non-extinction of 
contaminated $\sfA$-cells, that is $\Prob(\Ext(\sfA))<1$. Basic results on $\calZ_n(\sfA)$, $Z_{[n]}$, $\#\G^*_n(\sfA)$ and $\#\G^*_n$ including the afore-mentioned one will be shown in Section \ref{Section.Preliminaries} and be partly instrumental for the proofs of our results on the asymptotic behavior of the relative proportion of contaminated cells with $k$ parasites within the population of all contaminated cells. These results are stated in Section \ref{Section.Results} and proved in Section \ref{Section.Proofs}. A glossary of the most important notation used throughout may be found at the end of this article.

\section{Basic Results}\label{Section.Preliminaries}

We begin with a number of basic properties of and results about the quantities $\G^*_n(\sfA)$, $\G^*_n$, $\calZ_n(\sfA)$ and $Z_{[n]}$.

\subsection{The random $\sfA$-cell line and its associated sequence $(Z_{[n]})_{n\ge 0}$}

In \cite{Bansaye:08}, a random cell line was obtained by simply picking a random path in the infinite binary Ulam-Harris tree representing the cell population. Due to the multi-type structure here, we must proceed in a different manner when restricting to a specific cell type, here type $\sfA$. In order to study the properties of a ''typical'' $\sfA$-cell in generation $n$ for large $n$, i.e., an $\sfA$-cell picked at random from this generation, a convenient (but not the only) way is to first pick at random a cell line up to generation $n$ from the full height $n$ binary tree as in \cite{Bansaye:08} and then to condition upon the event that the cell picked at generation $n$ is of type $\sfA$. This naturally leads to a random $\sfA$-cell line up to generation $n$, for $\sfA$-cells can only stem from cells of the same type. Then looking at the conditional distribution of the associated parasitic random vector $(Z_{[0]},...,Z_{[n]})$ leads to a BPRE not depending on $n$ and thus to an analogous situation as in 
\cite{Bansaye:08}. The precise result is stated next.

\begin{Proposition}\label{prop:random cell line=BPRE}
Let $(Z_{n}(\sfA))_{n\ge 0}$ be a BPRE with one ancestor and iid environmental sequence
$(\Lambda_{n})_{n\ge 1}$ taking values in $\{\mathcal L(X^{(0)}(\sfA, \sfAA)), \mathcal L(X^{(1)}(\sfA, \sfAA)), \mathcal L(X^{(0)}(\sfA, \sfAB))\}$ such that
\begin{equation*}
 	\Prob\left(\Lambda_1=\mathcal L(X^{(0)}(\sfA, \sfAB))\right)=\frac{p_{\sfAB}}{\nu}\quad\text{and}\quad\Prob\left(\Lambda_1=\mathcal L(X^{(i)}(\sfA, \sfAA))\right)=\frac{p_{\sfAA}}{\nu},
\end{equation*}
for $i\in\{0,1\}$.
Then the conditional law of $(Z_{[0]},...,Z_{[n]})$ given $T_{[n]}=\sfA$ equals the law of $(Z_{0}(\sfA),...,Z_{n}(\sfA))$, for each $n\ge 0$.
\end{Proposition}

\begin{proof}
We use induction over $n$ and begin by noting that nothing has to be shown if $n=0$.
For $n\ge 1$ and $(z_0,...,z_n)\in\N^{n+1}_0$, we introduce the notation
\begin{equation*}
 	C_{z_0,...,z_n}:=\{(Z_{[0]},...,Z_{[n]})=(z_0,...,z_n)\}\quad\text{and}\quad C_{z_0,...,z_n}^{\sfA}:=C_{z_0,...,z_n}\cap\{T_{[n]}=\sfA\}
\end{equation*}
and note that
\begin{equation*}
  \Prob\left(T_{[n]}=\sfA\right)=2^{-n}\, \E\left(\sum_{v\in\G_n}\1_{\{T_v=\sfA\}}\right)=\left(\frac{\nu}{2}\right)^n,
\end{equation*}
for each $n\in\N$, in particluar
$$ \Prob(T_{[n]}=\sfA|T_{[n-1]}=\sfA)=\frac{\Prob(T_{[n]}=\sfA)}{\Prob(T_{[n-1]}=\sfA)}=\frac{\nu}{2}. $$
Assuming the assertion holds for $n-1$ (inductive hypothesis), thus
$$ \Prob(C_{z_{0},...,z_{n-1}}|T_{[n-1]}=\sfA)=\Prob\left(Z_{0}(\sfA)=z_0,...,Z_{n-1}(\sfA)=z_{n-1}\right) $$
for any $(z_{0},...,z_{n-1})\in\N_{0}^{n}$, we infer with the help of the Markov property that
\begin{align*}
 	\Prob&\left((Z_{[0]},...,Z_{[n]})=(z_0,...,z_n)| T_{[n]}=\sfA\right)\\[1ex]
	&=~\frac{\Prob(C^{\sfA}_{z_0,...,z_n})}{\Prob(T_{[n]}=\sfA)}\\[1ex]
	&=~\Prob\left(C_{z_0,...,z_{n-1}}|T_{[n-1]}=\sfA\right)\,
	\Prob(Z_{[n]}=z_n,T_{[n]}=\sfA|C^{\sfA}_{z_0,...,z_{n-1}})\,\frac{\Prob(T_{[n-1]}=
	\sfA)}{\Prob(T_{[n]}=\sfA)}\\[1ex]
	&=~\Prob\left(Z_{0}(\sfA)=z_0,...,Z_{n-1}(\sfA)=z_{n-1}\right)\,
	\frac{\Prob(Z_{[1]}=z_n, T_{[1]}=\sfA | Z_{[0]}=z_{n-1}, 
	T_{[0]}=\sfA)}{\Prob(T_{[n]}=\sfA|T_{[n-1]}=\sfA)}\\
	&=~\Prob\left(Z_{0}(\sfA)=z_0,...,Z_{n-1}(\sfA)=z_{n-1}\right)\\
	&\hspace{4ex}\times\frac{2}{\nu}\left(\frac{p_{\sfA\sfB}}{2}\left(\Prob^{X^{(0)}
	(\sfA, \sfAB)}\right)^{*z_{n-1}}(\{z_n\})+\sum_{i\in\{0,1\}}\frac{p_{\sfA\sfA}}
	{2}\left(\Prob^{X^{(i)}(\sfA, \sfAA)}\right)^{*z_{n-1}}(\{z_n\})\right)\\[1ex]
	&=~\Prob\left(Z_{0}(\sfA)=z_0,...,Z_{n-1}(\sfA)=z_{n-1}\right)\,\Prob\left(Z_
	{[n]}(\sfA)=z_{n}|Z_{[n-1]}(\sfA)=z_{n-1}\right)\\[1ex]
	&=~\Prob\left(Z_{0}(\sfA)=z_0,...,Z_{n}(\sfA)=z_{n}\right).
\end{align*}
This proves the assertion.
\end{proof}

The connection between the distribution of $Z_{n}(\sfA)$ and the expected number of $\sfA$-cells in generation $n$ with a specified number of parasites is stated in the next result.

\begin{Proposition}\label{Prop:dist of Zn(A)}
For all $n\in\N$ and $k\in\N_{0}$,
\begin{equation}\label{Eq.BPRE.EG}
\Prob\left(Z_n(\sfA)=k\right)=\nu^{-n}\,\E\left(\#\{v\in\G_n(\sfA):Z_v=k\}\right),
\end{equation}
in particular
\begin{equation}\label{Eq.BPRE.EG,k>0}
\Prob\left(Z_n(\sfA)>0\right)=\nu^{-n}\,\E\#\G_{n}^{*}(\sfA).
\end{equation}
\end{Proposition}

\begin{proof}
For all $n,k\in\N$, we find that
\begin{align*}
	\E\left(\#\{v\in\G_n(\sfA):Z_v=k\}\right)
	~&=~\sum_{v\in\G_n}\Prob(Z_v=k, T_v=\sfA)\\[1ex]
	&=~2^n\Prob\left(Z_{[n]}=k, T_{[n]}=\sfA\right)\\[1ex]
	&=~2^n\Prob(T_{[n]}=\sfA)\Prob\left(Z_{[n]}=k| T_{[n]}=\sfA\right)\\[1ex]
	&=~\nu^n\Prob\left(Z_{[n]}=k| T_{[n]}=\sfA\right)\\[1ex]
	&=~\nu^n\Prob\left(Z_n(\sfA)=k\right),
\end{align*}
and this proves the result.
\end{proof}

For $n\in\N$ and $s\in [0,1]$, let 
\begin{equation*}
 	f_n(s|\Lambda) := \E(s^{Z_n(\sfA)}|\Lambda)\quad\text{and}\quad f_n(s) := \E s^{Z_n(\sfA)}=\E f_{n}(s|\Lambda)
\end{equation*}
denote the quenched and annealed generating function of $Z_n(\sfA)$, respectively, where $\Lambda:=(\Lambda_{n})_{n\ge 1}$.
Then the theory of BPRE (see \cite{Athreya:71.1, Athreya:71.2, GeKeVa:03, Smith+Wilkinson:69} for more details) provides us with the following facts: For each $n\in\N$,
\begin{align*}
f_n(s|\Lambda)=g_{\Lambda_{1}}\circ...\circ g_{\Lambda_{n}}(s),
\quad g_{\lambda}(s):=\E(s^{Z_{1}(\sfA)}|\Lambda_{1}=\lambda)=\sum_{n\ge 0}\lambda_{n}s^{n}
\end{align*}
for any distribution $\lambda=(\lambda_{n})_{n\ge 0}$ on $\N_{0}$. Moreover, the $g_{\Lambda_{n}}$ are iid with
\begin{align*}
\E g_{\Lambda_{1}}'(1)&=\E Z_{1}(\sfA)\\
&=\frac{p_{\sfAA}}{\nu}\Big(\mu_{0,\sfA}(\sfAA)+\mu_{1,\sfA}(\sfAA)\Big)+\frac{p_{\sfAB}}{\nu}\mu_{0,\sfA}(\sfAB)=\frac{\gamma}{\nu},
\end{align*}
where
\begin{equation*}
\gamma := \E\calZ_1(\sfA) = p_{\sfAA}\left(\mu_{0,\sfA}(\sfAA)+\mu_{1,\sfA}(\sfAA)\right)+p_{\sfAB}\mu_{0,\sfA}(\sfAB) 
\end{equation*}
denotes the expected total number of parasites in cells of type $\sfA$ in the first generation (recall from \eqref{SA1} that $Z_{\varnothing}=Z_{\varnothing}(\sfA)=1$). As a consequence,
\begin{align*}
	\E(Z_{[n]}|T_{[n]}=\sfA)&=\E Z_n(\sfA)
	=f_n'(1)=\prod_{k=1}^{n}\E g_{\Lambda_{k}}'(1)=\left(\frac{\gamma}{\nu}\right)^n
\end{align*}
for each $n\in\N$. It is also well-known that $(Z_{n}(A))_{n\ge 0}$ dies out a.s., which in terms of $(Z_{[n]})_{n\ge 0}$ means that $\lim_{n\to\infty}\Prob(Z_{[n]}=0|T_{[n]}=\sfA)=1$, iff
\begin{align}\label{Eq.BPRE.Aussterben}
	\E\log g_{\Lambda_{1}}'(1)=\frac{p_{\sfAA}}{\nu}\Big(\log\mu_{0,\sfA}(\sfAA)+\log
	\mu_{1,\sfA}(\sfAA)\Big)+\frac{p_{\sfAB}}{\nu}\log\mu_{0,\sfA}(\sfAB)\le 0.
\end{align}

\subsection{Properties of $\#\G^*_n(\sfA)$ and $\#\G^*_n$:}
We proceed to the statement of a number of results on the asymptotic behavior of $\G^*_n(\sfA)$ and $\G^*_n$ conditioned upon $\Surv(\sfA)$ and $\Surv$, respectively. It turns out that, if the number of parasites tends to infinity, then so does the number of contaminated cells.

\begin{Theorem}\label{Satz.ExplosionInfZellen} ${}$
\begin{enumerate}[(a)]
\item\label{Item.ExplosionInfZellenB} If $\Prob(\Surv(\sfA))>0$ and $p_{\sfAA}>0$, then $\Prob(\#\G^*_n(\sfA)\to\infty|\Surv(\sfA))=1$.
\item\label{Item.ExplosionInfZellenAlle} If $\Prob(\Surv)>0$, then $\Prob(\#\G^*_n\to\infty|\Surv)=1.$
\end{enumerate}
\end{Theorem}

\begin{proof}
The proof of assertion \eqref{Item.ExplosionInfZellenB} is the same as for Theorem 4.1 in \cite{Bansaye:08} and thus omitted. 

\vspace{.2cm}
\eqref{Item.ExplosionInfZellenAlle} We first note that, given $\Surv$, a contaminated $\sfB$-cell is eventually created with probability one and then spawns a single-type cell process (as $\E\calZ_1(\sfB)>0$ by \eqref{SA5}). Hence the assertion follows again from Theorem 4.1 in \cite{Bansaye:08} if $\mu_{\sfB}:=\mu_{0,\sfB}+\mu_{1,\sfB}>1$.

Left with the case $\mu_{\sfB}\le 1$, it follows that
$$ \Prob(\Surv(\sfA)|\Surv)=1,$$
for otherwise, given $\Surv$, only $\sfB$-parasites would eventually be left with positive probability which however would die out almost surely. Next, $p_{\sfAA}>0$ leads back to  
\eqref{Item.ExplosionInfZellenB} so that it remains to consider the situation when $p_{\sfAA}=0$. In this case there is a single line of $\sfA$-cells, namely $\varnothing\to 0\to 00\to ...$, and $(\calZ_n(\sfA))_{n\ge 0}$ is an ordinary Galton-Watson branching process 
tending $\Prob(\cdot|\Surv(\sfA))$-a.s.\ to infinity. For $n,k\in\N$, let
\begin{equation*}
	\calZ_k(n,\sfB) := \sum_{v\in\G_{n+k+1}(\sfB):v|n+1=0^{n}1}Z_v
\end{equation*}
denote the number of $\sfB$-parasites at generation $k$ sitting in cells of the subpopulation stemming from the cell $0^{n}1$, where $0^n:=0...0$ ($n$-times). Using $p_{\sfAB}=1$ and \eqref{SA5}, notably $\mu_{1,\sfA}(\sfAB)>0,\mu_{0,\sfB}>0$ and $\mu_{1,\sfB}>0$, it
is readily seen that
	\begin{equation*}
	 	\Prob\left(\lim_{n\to\infty}\calZ_0(n-k,\sfB)=\infty|\Surv(\sfA)\right)=1
	\end{equation*}
and thus
	\begin{equation*}
	 	\Prob\left(\lim_{n\to\infty}\calZ_K(n-k,\sfB)=0|\Surv(\sfA)\right)=0
	\end{equation*}
for all $K\in\N$ and $k\leq K$. Consequently,
	\begin{align*}
	 	&\Prob\left( \liminf_{n\to\infty}\#\G^*_n\leq K|\Surv(\sfA)\right)\\
		&\hspace{1cm}\leq\ \Prob\left(\lim_{n\to\infty}\max_{0\le k\le K}
		\calZ_k(n-k,\sfB)=0|\Surv(\sfA)\right)\\
		&\hspace{1cm}\leq\ \sum_{k=0}^{K}
		\Prob\left(\lim_{n\to\infty}\calZ_K(n-k,\sfB)=0|\Surv(\sfA)\right)\\
		&\hspace{1cm}=~0
	\end{align*}
for all $K\in\N$
\end{proof}

The next result provides us with the geometric rate at which the number of contaminated cells tends to infinity.

\begin{Theorem}\label{Satz.ErholungB} $(\nu^{-n}\#\G^*_n(\sfA))_{n\ge 0}$ is a non-negative supermartingale and therefore a.s.\ convergent to a random variable $L(\sfA)$ as $n\to\infty$. Furthermore, 
 \begin{enumerate}[(a)]
	\item\label{Item.ErholungB.FSExtinction} $L(\sfA)=0$ a.s.\ iff $\E\log g_{\Lambda_
	{1}}'(1)\leq 0$ or $\nu\leq1$
	\item\label{Item.ErholungB.Extinction} $\Prob(L(\sfA)=0)<1$ implies $\{L(\sfA)=0\}=
	\Ext(\sfA)$ a.s.
\end{enumerate}
\end{Theorem}

\begin{proof}
That $(\nu^{-n}\#\G^*_n(\sfA))_{n\ge 0}$ forms a supermartingale follows by an easy calculation and therefore a.s. convergence to an integrable random variable $L(\sfA)$ is ensured. This supermartingale is even uniformaly integrable in the case $\nu>1$, which follows because the obvious majorant $(\nu^{-n}\#\G_n(\sfA))_{n\ge 0}$ is a normalized Galton-Watson branching process having a reproduction law with finite variance and is thus $L^2$-bounded (see Section I.6 in \cite{Athreya+Ney:72}). Consequently,  $(\nu^{-n}\#\G^*_n(\sfA))_{n\geq0}$ is uniformaly integrable and
\begin{align}\label{eq:EL(A)}
 	\E L(\sfA) = \lim_{n\to\infty}\E\frac{\#\G_{n}^{*}(\sfA)}{\nu^n}
		= \lim_{n\to\infty}\Prob(Z_n(\sfA)>0),
\end{align}
the last equality following from \eqref{Eq.BPRE.EG,k>0} in Proposition \ref{Prop:dist of Zn(A)}.

\vspace{.2cm}	
As for \eqref{Item.ErholungB.FSExtinction}, $L(\sfA)=0$ a.s.\ occurs iff either $\nu\le 1$, in which case $\#\G_n^*(\sfA)\le\#\G_{n}(\sfA)=0$ eventually, or $\nu>1$ and $\E\log g_{\Lambda_{1}}'(1)\leq 0$, in which case almost certain extinction of $(Z_{n}(\sfA))_{n\ge 0}$ in combination with \eqref{eq:EL(A)} yields the conclusion.

\vspace{.2cm}
\eqref{Item.ErholungB.Extinction} Defining $\tau_n = \inf\{m\in\N:\#\G^*_m(\sfA)\geq n\}$,
we find that
\begin{align*}
	 \Prob(L(\sfA)=0) &\leq~ \Prob(L(\sfA)=0|\tau_n<\infty)+\Prob(\tau_n=\infty)\\
	&\leq\ \Prob\left(\bigcap_{k=1}^{\#\G^*_{\tau_n}(\sfA)}\{\#\G^*_{m,k}(\sfA)/\nu^m
	\to0\}\bigg|\tau_n<\infty\right)+\Prob(\tau_n=\infty)\\
	&\leq\ \Prob(L(\sfA)=0)^n+\Prob(\tau_n=\infty)
\end{align*}
for all $n\ge 1$, where the $\#\G^*_{m,k}(\sfA)$, $k\ge 1$, are independent copies of $\#\G^*_m(\sfA)$. Since $\Prob(L(\sfA)=0)<1$, Theorem \ref{Satz.ExplosionInfZellen} implies
\begin{equation*}
	 \Prob(L(\sfA)=0)\leq \lim_{n\to\infty}\Prob(\tau_n=\infty)=\Prob\left(\sup_{n\ge 1}
	 \#\G^*_{n}(\sfA)<\infty\right)=\Prob(\Ext(\sfA))
	\end{equation*}
which in combination with $\Ext(\sfA)\subset\{L(\sfA)=0\}$ a.s. proves the assertion.
\end{proof}

Since $\nu<2$ and $(\nu^{-n}\#\G_{n}(\sfA))_{n\ge 0}$ is a nonnegative, a.s.\ convergent martingale, we see that $2^{-n}\#\G^{*}_n(\sfA)\le 2^{-n}\#\G_n(\sfA)\to 0$ a.s. and therefore
\begin{equation*}
 	\frac{\#\G^*_n}{2^n}\ \simeq\ \frac{\#\G^*_n(\sfB)}{2^n},\quad\text{as }n\to\infty.
\end{equation*}
that is, the asymptotic proportion of all contaminated cells is the same as the asymptotic proportion of contaminated $\sfB$-cells. Note also that
\begin{equation}\label{eq:T[n]to zero}
\Prob(T_{[n]}=\sfA)=\E\left(\frac{\#\G_n(\sfA)}{2^n}\right)\to 0,\quad\text{as }n\to\infty.
\end{equation}
Further information is provided by the next result.

\begin{Theorem}\label{Satz.Erholung}
There exists a r.v. $L\in[0,1]$ such that $\#\G^*_n/2^n\to L$ a.s. Furthermore,
\begin{enumerate}[(a)]
	\item\label{Item.Erholung.AlleFSExtinction} $L=0$ a.s. iff 
	$\mu_{0,\sfB}\mu_{1,\sfB}\leq1$.
	\item\label{Item.Erholung.AlleExtinction} If $\Prob(L=0)<1$, then $\{L=0\}=\Ext$ 
	a.s.
\end{enumerate}
\end{Theorem}

\begin{proof}
The existence of $L$ follows because $2^{-n}\#\G^*_n$ is obviously decreasing. 
As for \eqref{Item.Erholung.AlleFSExtinction}, suppose first that $\mu_{0,\sfB}\mu_{1,\sfB}\leq1$ and note that this is equivalent to almost sure extinction of a random $\sfB$-cell line, i.e.
$$ \lim_{n\to\infty}\Prob(Z_{[n]}>0|Z_{\varnothing}=k,T_{[0]}=\sfB)=0 $$
for any $k\in\N$. This follows because, starting from a $\sfB$-cell, we are in the one-type model studied in \cite{Bansaye:08}. There it is stated that $(Z_{[n]})_{n\geq0}$ forms a BPRE which dies out a.s.\ iff $\mu_{0,\sfB}\mu_{1,\sfB}\leq1$ (see \cite[Prop.\ 2.1]{Bansaye:08}). Fix any $\varepsilon>0$ and choose $m\in\N$ so large that $\Prob(T_{[m]}=\sfA)\leq\varepsilon$, which is possible by \eqref{eq:T[n]to zero}. Then, by the monotone convergence theorem, we find that for sufficiently large $K\in\N$
	\begin{align*}
	 	\E L ~&=~ \lim_{n\to\infty}\Prob(Z_{[n+m]}>0)\\
		&\leq~ \lim_{n\to\infty}\Prob(Z_{[n+m]}>0, T_{[m]}=\sfB) + \varepsilon\\
		&\leq~ \lim_{n\to\infty}\sum_{k\ge 0}\Prob(Z_{[n+m]}>0, Z_{[m]}=k, 
		T_{[m]}=\sfB) + \varepsilon\\
		&\leq~ \lim_{n\to\infty}\sum_{k=0}^{K}\Prob(Z_{[n]}>0|Z_{[0]}=k, 
		T_{[0]}=\sfB) + 2\varepsilon\\
		&\leq~ 2\varepsilon
	\end{align*}
and thus $\E L=0$. For the converse, note that
\begin{align*}
	 	0 ~&=~ \E L\\
		&=~ \lim_{n\to\infty}\Prob(Z_{[n+1]}>0)\\
		&\geq~ \lim_{n\to\infty}\Prob(Z_{[1]}>0, T_{[1]}=\sfB)\Prob(Z_{[n]}>0|T_{[0]}=\sfB)
	\end{align*}
implies $0=\lim_{n\to\infty}\Prob(Z_{[n]}>0|T_{[0]}=\sfB)$ and thus $\mu_{0,\sfB}\mu_{1,\sfB}\leq1$ as well.

\vspace{.2cm}
The proof of \eqref{Item.Erholung.AlleExtinction} follows along similar lines as Theorem \ref{Satz.ErholungB}\eqref{Item.ErholungB.Extinction} and is therefore omitted.
\end{proof}

\subsection{Properties of $\calZ_n(\sfA)$}

We continue with some results on $\calZ_n(\sfA)$, the number of $\sfA$-parasites at generation $n$, and point out first that $(\gamma^{-n}\calZ_n(\sfA))_{n\ge 0}$ constitutes a nonnegative, mean one martingale which is a.s.\ convergent to a finite random variable $W$. In particular, $\E\calZ_n(\sfA)=\gamma^n$ for all $n\in\N_{0}$. If $\E\calZ_1(\sfA)^2<\infty$, $\gamma>1$ and
$$ \hat\gamma\ :=\ \nu\,\E g^{\prime}_{\Lambda_1}(1)^2=p_{\sfAA}\left(\mu^2_{0,\sfA}(\sfAA)+\mu^2_{1,\sfA}(\sfAA)\right)+p_{\sfAB}\mu^2_{0,\sfA}(\sfAB)\ \leq\ \gamma, $$
then the martingale is further $L^{2}$-bounded as may be assessed by a straightforward but tedious computation. The main difference between a standard Galton-Watson process and the $\sfA$-parasite process $(\calZ_n(\sfA))_{n\geq0}$ is the dependence of the offspring numbers of parasites living in the same cell, which (by some elementary calculations) leads to an additional term in the recursive formula for the variance, viz. 
\begin{equation*}
	\Var\left(\calZ_{n+1}(\sfA)\right)~=~\gamma^2\,\Var\left(\calZ_{n}(\sfA)\right)+\gamma^n\,\Var(\calZ_1(\sfA))+c_{1} \nu^n f_{n}''(1)
\end{equation*}
for all $n\geq0$ and some finite positive constant $c_{1}$. Here it should be recalled that $f_n(s)=\E s^{Z_n(\sfA)}$. Consequently, by calculating the second derivative of $f_{n}$ and using $\hat\gamma\le\gamma$, we obtain
\begin{equation*}
 f^{\prime\prime}_n(1)\ =\ \E g^{\prime\prime}_{\Lambda_1}(1) \sum_{i=1}^n\left(\frac{\hat\gamma}{\nu}\right)^{n-i}\left(\frac{\gamma}{\nu}\right)^{i-1}\ \leq\ c_{2} n \left(\frac{\gamma}{\nu}\right)^n
\end{equation*}
for some finite positive constant $c_2$. A combination of this inequality with the above recursion for the variance of $\calZ_{n}(\sfA)$ finally provides us with
\begin{equation*}
	\Var\left(\gamma^{-n}\calZ_{n}(\sfA)\right)~\leq~1+\gamma^{-2}\sum_{k=0}^{\infty}\gamma^{-k}\big(\Var(\calZ_1(\sfA))+c_{1} c_{2} k\big)\ <\ \infty
\end{equation*}
for all $n\ge 0$ and thus the $L^2$-boundness of $(\gamma^{-n}\calZ_{n}(\sfA))_{n\geq0}$.


\vspace{.2cm}

Recalling that $(\calZ_n(\sfA))_{n\ge 0}$ and $(\calZ_n)_{n\ge 0}$ satisfy the extinction-explosion principle, the next theorem gives conditions for almost sure extinction, that is, for
$\Prob(\Ext(\sfA))=1$ and $\Prob(\Ext)=1$. 

\begin{Theorem}\label{Satz.Aussterben}
\begin{enumerate}[(a)]
	\item\label{Extinction_p=0} If $p_{\sfAA}=0$, then
		\begin{equation*}
 			\Prob(\Ext(\sfA))=1\quad\text{iff}\quad \mu_{0,\sfA}(\sfAB)\le 1\quad\text{or}\quad\nu<1.
		\end{equation*}
	\item\label{Extinction_p>0} If $p_{\sfAA}>0$, then the following statements are 
	equivalent:
	\begin{enumerate}[(1)]\setlength{\itemsep}{1ex}
 		\item\label{Item.Extinction.1} $\Prob(\Ext(\sfA))=1$
		\item\label{Item.Extinction.EG} $\E\#\G^*_n(\sfA)\leq1$ for 
		all $n\in\N\quad$
		\item\label{Item.Extinction.gamma} $\nu\leq1$, or
			\begin{equation*}
			\nu>1,\quad\E\log g_{\Lambda_{1}}'(1)<0\quad\text{and}\quad\inf_{0\leq\theta
			\leq1}\E g_{\Lambda_{1}}'(1)^{\theta}\leq \frac{1}{\nu}.
			\end{equation*}
		\end{enumerate}
	\item\label{Extinction_all} $\Prob(\Ext)=1\quad\text{iff}\quad \Prob(\Ext
	(\sfA))=1\ \text{and}\ \mu_{0,\sfB}+\mu_{1,\sfB}\leq1$
\end{enumerate}
\end{Theorem}

\begin{Remark}
Let us point out the following useful facts before proceeding to the proof of the theorem. We first note that, if $\E\log g_{\Lambda_{1}}'(1)<0$ and $\E g_{\Lambda_{1}}'(1)\log g_{\Lambda_{1}}'(1)\leq 0$, then the convexity of
$\theta\mapsto\E g_{\Lambda_{1}}'(1)^{\theta}$ implies that
\begin{equation*}
 	\E g_{\Lambda_{1}}'(1)=\inf_{0\leq\theta\leq1}\E g_{\Lambda_{1}}'(1)^{\theta}.
\end{equation*}
If $\E Z_{1}(\sfA)^{2}<\infty$, Geiger et al. \cite[Theorems 1.1--1.3]{GeKeVa:03} showed that
\begin{equation}\label{eq:Geiger survival estimate}
\Prob(Z_{n}(\sfA)>0)\simeq cn^{-\kappa}\left(\inf_{0\leq\theta\leq 1}\E g_{\Lambda_{1}}'(1)^{\theta}\right)^{n}\quad\text{as }n\to\infty
\end{equation}
for some $c\in (0,\infty)$, where
\begin{itemize}
\item[] $\kappa=0$ if $\E g_{\Lambda_{1}}'(1)\log g_{\Lambda_{1}}'(1)<0$,\hfill (strongly subcritical case)
\item[] $\kappa=1/2$ if $\E g_{\Lambda_{1}}'(1)\log g_{\Lambda_{1}}'(1)=0$,\hfill (intermediately subcritical case)
\item[] $\kappa=3/2$ if $\E g_{\Lambda_{1}}'(1)\log g_{\Lambda_{1}}'(1)>0$,\hfill (weakly subcritical case)
\end{itemize}
A combination of \eqref{Eq.BPRE.EG,k>0} and \eqref{eq:Geiger survival estimate} provides us with the asymptotic relation
\begin{equation}\label{eq:Gn and Geiger}
\E\#\G^*_{n}(\sfA)\simeq cn^{-\kappa}\nu^{n}\left(\inf_{0\leq\theta\leq 1}\E g_{\Lambda_{1}}'(1)^{\theta}\right)^{n}\quad\text{as }n\to\infty,
\end{equation}
in particular (with $\E Z_{1}(\sfA)^{2}<\infty$ still being in force)
\begin{equation}\label{eq:Gn and Geiger2}
\inf_{0\leq\theta\leq 1}\E g_{\Lambda_{1}}'(1)^{\theta}\le\frac{1}{\nu}\quad\text{if}\quad\sup_{n\ge 1}\E\#\G^*_{n}(\sfA)<\infty.
\end{equation}
\end{Remark}

\begin{proof}
\eqref{Extinction_p=0} If $p_{\sfAA}=0$ and $\nu=p_{\sfAB}=1$, each generation possesses exactly one $\sfA$-cell and $(\calZ_n(\sfA))_{n\ge 0}$ thus forms a Galton-Watson branching process with offspring mean $\mu_{0,\sfA}(\sfAB)$ and positive offspring variance (by \eqref{SA3}). Hence a.s.\ extinction occurs iff $\mu_{0,\sfA}(\sfAB)\le 1$ as claimed. If $\nu<1$, type $\sfA$ cells die out a.s. and so do type $\sfA$ parasites.


\vspace{.2cm}
``\eqref{Item.Extinction.1}$\Rightarrow$\eqref{Item.Extinction.EG}''
(by contraposition) We fix $m\in\N$ such that $\E\left(\#\G^*_m(\sfA)\right)>1$ and consider a supercritical Galton-Watson branching process $(S_n)_{n\ge 0}$ with $S_0=1$ and offspring distribution
\begin{equation*}
 	\Prob(S_1=k) = \Prob(\#\G^*_m(\sfA)=k),\quad k\in\N_{0}.
\end{equation*}
Obviously,
\begin{equation*}
 	\Prob(S_n>k)\leq \Prob(\#\G^*_{nm}(\sfA)>k)
\end{equation*}
for all $k,n\in\N_0$, hence 
\begin{equation*}
 	\lim_{n\to\infty}\Prob(\#\G^*_{nm}(\sfA)>0)\geq\lim_{n\to\infty}\Prob(S_n>0)>0,
\end{equation*}
i.e.\ $\sfA$-parasites survive with positive probability.

\vspace{.2cm}
``\eqref{Item.Extinction.EG}$\Rightarrow$\eqref{Item.Extinction.1}''
If $\E\#\G^*_n(\sfA)\leq1$ for all $n\in\N$, then Fatou's lemma implies
\begin{equation*}
 	1\geq\liminf_{n\to\infty} \E\#\G^*_n(\sfA)\geq \E\left(\liminf_{n\to\infty}\#\G^*_n
	(\sfA)\right)
\end{equation*}
giving $\Prob(\Ext(\sfA))=1$ by an appeal to Theorem \ref{Satz.ExplosionInfZellen}.

\vspace{.2cm}
``\eqref{Item.Extinction.gamma}$\Rightarrow$\eqref{Item.Extinction.1},\eqref{Item.Extinction.EG}''
If $\nu\leq1$ then $\E\#\G_{n}^{*}(\sfA)\le\E\#\G_n(\sfA)=\nu^{n}\le 1$ for all $n\in\N$. So let us consider the situation when
$$\nu>1,\quad\E\log g_{\Lambda_{1}}'(1)<0\quad\text{and}\quad\inf_{0\leq\theta\leq1}\E g_{\Lambda_{1}}'(1)^{\theta}\leq\frac{1}{\nu} $$
is valid. By \eqref{Eq.BPRE.EG,k>0},
\begin{equation*}
 	\E\#\G^*_n(\sfA) = \nu^n \Prob(Z_n(\sfA)>0)
\end{equation*}
for all $n\in\N$. We must distinguish three cases:

\vspace{.2cm} \textsc{Case A}. $\E g_{\Lambda_{1}}'(1)\log g_{\Lambda_{1}}'(1)\leq 0$.
By what has been pointed out in the above remark, we then infer
\begin{equation*}
 	\frac{\gamma}{\nu}=\E g_{\Lambda_{1}}'(1)=\inf_{0\leq\theta\leq1}\E 
	g_{\Lambda_{1}}'(1)^{\theta}\le\frac{1}{\nu}
\end{equation*}
and thus $\gamma\le 1$, which in turn entails
\begin{equation*}
	\E\#\G^*_n(\sfA)\leq \E\calZ_n(\sfA) = \gamma^n \leq 1
\end{equation*}
for all $n\in\N$ as required.

\vspace{.2cm} 
\textsc{Case B}. $\E g_{\Lambda_{1}}'(1)\log g_{\Lambda_{1}}'(1)>0$ and $\E Z_{1}(\sfA)^{2}<\infty$. Then, by \eqref{eq:Geiger survival estimate},
\begin{equation*}
 	\Prob(Z_n(\sfA)>0)\ \simeq\ c n^{-3/2} \left(\inf_{0\leq\theta\leq 1}
	\E g_{\Lambda_{1}}'(1)^{\theta}\right)^n\quad\text{as } n\to\infty
\end{equation*}
holds true for a suitable constant $c\in (0,\infty)$ and therefore
\begin{equation*}
	0 = \lim_{n\to\infty} \nu^n \Prob(Z_n(\sfA)>0) = \liminf_{n\to\infty} \E\#\G^*_n(\sfA)
	\geq \E\left(\liminf_{n\to\infty}\#\G^*_n(\sfA)\right),
\end{equation*}
implying $\Prob(\Ext(\sfA))=1$ by Theorem \ref{Satz.ExplosionInfZellen}.

\vspace{.2cm} 
\textsc{Case C}. $\E g_{\Lambda_{1}}'(1)\log g_{\Lambda_{1}}'(1)>0$ and $\E Z_{1}(\sfA)^{2}=\infty$. Using contraposition, suppose that $\sup_{n\in\N}\E\#\G^*_n(\sfA)>1$.
Fix any vector $\alpha=(\alpha^{(u)}_{\sfs})_{u\in\{0,1\},\sfs\in\{\sfAA,\sfAB,\sfBB\}}$
of distributions on $\N_{0}$ satisfying
$$ \alpha_{\sfs,x}^{(u)}\ \le\ \Prob\left(X^{(u)}_{1,v}(\sfA,\sfs)=x\right)\quad\text{for }x\geq 1 $$
and $u,\sfs$ as stated, hence
$$ \alpha_{\sfs,0}^{(u)}\ \ge\ \Prob\left(X^{(u)}_{1,v}(\sfA,\sfs)=0\right)\quad\text{and}\quad\sum_{x\ge n}\alpha_{\sfs,x}^{(u)}\ \le\ \Prob\left(X^{(u)}_{1,v}(\sfA,\sfs)\ge n\right) $$
for each $n\ge 0$. Possibly after enlarging the underlying probability space, we can then construct a cell division process $(Z_{\alpha,v}, T_v)_{v\in\T}$ coupled with and of the same kind as $(Z_{v}, T_v)_{v\in\T}$ such that
\begin{align*}
&X^{(u)}_{\alpha, k,v}(\sfA,\sfs)\ \leq\ X^{(u)}_{k,v}(\sfA,\sfs)\quad\text{a.s.}\\
\text{and}\quad
&\Prob\left(X^{(u)}_{\alpha,k,v}(\sfA,\sfs) = x \right)\ =\ \alpha_{\sfs,x}^{(u)}
\end{align*}
for each $u\in\{0,1\}$, $\sfs\in\{\sfAA,\sfAB,\sfBB\}$, $v\in\T$, $k\geq1$ and $x\ge 1$.
To have $(Z_{\alpha,v}, T_v)_{v\in\T}$ completely defined, put also 
$$ (X^{(0)}_{\alpha,k,v}(\sfB),X^{(1)}_{\alpha,k,v}(\sfB)):=(X^{(0)}_{k,v}(\sfB),X^{(1)}_{k,v}(\sfB)) $$
for all $v\in\T$ and $k\ge 1$. Then $Z_{\alpha,v}\leq Z_v$ a.s.\ and thus
\begin{equation}\label{Eq.GestutzterProzess.G}
 	\E g_{\alpha,\Lambda_{1}}'(1)^{\theta}\ \leq\ \E g_{\Lambda_{1}}'(1)^{\theta},\quad \theta
	\in[0,1],
\end{equation}
where $Z_{\alpha,k}(\sfA)$ and $g_{\alpha,\Lambda_{1}}$ have the obvious meaning. Since the choice of $\alpha$ has no affect on the cell splitting process, we have $\nu_{\alpha}=\nu>1$, while \eqref{Eq.GestutzterProzess.G} ensures
\begin{equation}\label{eq:truncation gf}
\E\log g_{\alpha,\Lambda_{1}}'(1)\le\E\log g_{\Lambda_{1}}'(1)<0.
\end{equation}
For $N\in\N$ let $\alpha(N)=(\alpha_{\sfs}^{(u)}(N))_{u\in\{0,1\},\sfs\in\{\sfAA,\sfAB,\sfBB\}}$ be the vector specified by
\begin{equation*}
 \alpha_{\sfs,x}^{(u)}(N)\ :=\  
\begin{cases}
 \Prob\left(X^{(u)}_{k,v}(\sfA,\sfs) = x \right), &\text{if $1\leq x\leq N$}\\
 0, &\text{if $x>N$}.
\end{cases}
\end{equation*}
Then $\E Z_{\alpha(N),1}(\sfA)^2<\infty$ and we can fix $N\in\N$ such that $\sup_{n\in\N}\E\#\G^*_{\alpha(N),n}(\sfA)>1$, because $\#\G^*_{\alpha(N),n}(\sfA)\uparrow\#\G^*_{n}(\sfA)$ as $N\to\infty$. Then, by what has already been proved under Case B in combination with \eqref{Eq.GestutzterProzess.G},\eqref{eq:truncation gf} and $\nu_{\alpha(N)}>1$, we infer
\begin{equation*}
	\inf_{0\leq\theta\leq1}\E g_{\Lambda_{1}}'(1)^{\theta}\geq\inf_{0\leq\theta\leq1}\E g_{\alpha(N),\Lambda_{1}}'(1)^{\theta}>\frac{1}{\nu}.
\end{equation*}
and thus violation of \eqref{Item.Extinction.EG}.

\vspace{.2cm} 
``\eqref{Item.Extinction.EG}$\Rightarrow$\eqref{Item.Extinction.gamma}''
Suppose $\E\#\G_{n}^{*}(\sfA)\le 1$ for all $n\in\N$ and further $\nu>1$ which, by 
\eqref{Eq.BPRE.EG,k>0}, entails $\lim_{n\to\infty}\Prob(Z_{n}(\sfA)>0)=0$ and thus $\E\log g_{\Lambda_{1}}'(1)\le 0$. We must show that $\E\log g_{\Lambda_{1}}'(1)<0$ and $\inf_{0\leq\theta\leq1}\E g_{\Lambda_{1}}'(1)^{\theta}\le\nu^{-1}$. But given $\E\log g_{\Lambda_{1}}'(1)<0$, the second condition follows from \eqref{eq:Gn and Geiger2} if $\E Z_{1}(\sfA)^{2}<\infty$, and by a suitable ``$\alpha$-coupling'' as described under Case C above if $\E Z_{1}(\sfA)^{2}=\infty$. Hence it remains to rule out that $\E\log g_{\Lambda_{1}}'(1)=0$. Assuming the latter, we find with the help of Jensen's inequality that
\begin{equation*}
 	\inf_{0\leq\theta\leq 1}\log \E g_{\Lambda_{1}}'(1)^{\theta}\geq \inf_{0\leq\theta\leq1}\theta \,\E\log g_{\Lambda_{1}}'(1) = 0
\end{equation*}
or, equivalently,
\begin{equation*}
	\inf_{0\leq\theta\leq1}\E g_{\Lambda_{1}}'(1)^{\theta}\geq 1>\frac{1}{\nu}
\end{equation*}
(which implies $\inf_{0\leq\theta\leq1}\E g_{\Lambda_{1}}'(1)^{\theta}=1$).
Use once more a suitable ``$\alpha$-coupling'' and fix $\alpha$ in such a way that
\begin{equation*}
 1=\inf_{0\leq\theta\leq1}\E g_{\Lambda_{1}}'(1)^{\theta} >\inf_{0\leq\theta\leq1}\E g_{\alpha,\Lambda_{1}}'(1)^{\theta}>\frac{1}{\nu}
\end{equation*}
which implies subcriticality of the associated BPRE $(Z_{\alpha,n}(\sfA))_{n\geq0}$.
By another appeal to \eqref{eq:Gn and Geiger2}, we thus arrive at the contradiction
\begin{equation*}
 	\sup_{n\in\N}\E\#\G^*_{n}(\sfA)\ge\sup_{n\in\N}\E\#\G^*_{\alpha,n}(\sfA)=\infty.
\end{equation*}
This completes the proof of \eqref{Extinction_p>0}.

\vspace{.2cm}
\eqref{Extinction_all}
Since $\Ext\subseteq \Ext(\sfA)$, we see that $\Prob(\Ext)=1$ holds iff $\Prob(\Ext(\sfA))=1$
and the population of $\sfB$-parasites dies out a.s.\ as well. But the latter form a Galton-Watson branching process with offspring mean $\mu_{0,\sfB}+\mu_{1,\sfB}$ once all $\sfA$-parasites have disappeared and hence die out as well iff $\mu_{0,\sfB}+\mu_{1,\sfB}\leq1$.
\end{proof}


\begin{Theorem}\label{Satz.W}
Assuming $\Prob(\Surv(\sfA))>0$ and thus particularly $\gamma>1$, the following assertions hold true:
 \begin{enumerate}[(a)]
 	\item\label{Item.W>0} If $\E\calZ_1(\sfA)^2<\infty$ and $\hat{\gamma}\leq\gamma$, then $\Prob(W>0)>0$ and $\E W=1$.
	\item\label{Item.W=0} If $\Prob(W=0)<1$, then $\Ext(\sfA)=\{W=0\}$ a.s.
 \end{enumerate}
\end{Theorem}

\begin{proof}
\eqref{Item.W>0}
As pointed out at the beginning of this subsection, $(\calZ_n(\sfA)/\gamma^n)_n$ is a $L^2$-bounded martingale and thus uniformly integrable. It therefore converges in $L^1$ to its limit $W$ satisfying $\E W=1$ as well as $\Prob(W>0)>0$.

\vspace{.1cm}
\eqref{Item.W=0} follows in the same manner as Theorem \ref{Satz.ErholungB}\eqref{Item.ErholungB.Extinction}.
\end{proof}

\section{Relative proportions of contaminated cells}\label{Section.Results}

We now turn to a statement of our main results that are concerned with the long-run behavior of relative proportions of contaminated cells containing a given number of parasites, viz.
\begin{equation*}
 	F_k(n):=\frac{\#\{v\in\G^*_n|Z_v=k\}}{\#\G^*_n}
\end{equation*}
for $k\in\N$ and $n\to\infty$, and of the corresponding quantities when restricting to contaminated cells of a given type $\sft$, viz.
\begin{equation*}
 	F_k(n,\sft):=\frac{\#\{v\in\G^*_n(\sft)|Z_v=k\}}{\#\G^*_n(\sft)}
	\end{equation*}
for $\sft\in\{\sfA,\sfB\}$. Note that
\begin{equation*}
 	F_k(n) = F_k(n,\sfA)\,\frac{\#\G^*_n(\sfA)}{\#\G^*_n}+F_k(n,\sfB)\,\frac{\#\G^*_n		(\sfB)}{\#\G^*_n}.
\end{equation*}
Given survival of type-$\sfA$ parasites, i.e.\ conditioned upon the event $\Surv(\sfA)$, our results, devoted to regimes where at least one of $\sfA$- or $\sfB$-parasites multiply at a high rate, describe the limit behavior of $F_k(n,\sfA)$, $\#\G^*_n(\sfA)/\#\G^*_n$ and $F_k(n,\sfB)$, which depends on that of $\calZ_n(\sfA)$ and the BPRE $Z_n(\sfA)$ in a crucial way.

\vspace{.2cm}
For convenience, we define
\begin{equation*}
  \Prob_{z,\sft} := \Prob(\cdot|Z_{\varnothing}=z, T_{\varnothing}=\sft),\quad z\in\N,\ \sft\in\{\sfA,\sfB\},
\end{equation*}
and use $\E_{z,\sft}$ for expectation under $\Prob_{z,\sft}$. Recalling that $\Prob$ stands for $\Prob_{1,\sfA}$, 
we put $\Prob^{*}:=\Prob(\cdot|\Surv(\sfA))$ and, furthermore,
\begin{equation*}
 	\Prob^*_{z,\sft}:=\Prob_{z,\sft}(\cdot|\Surv(\sfA))\quad \text{and}\quad \Prob^n_{z,\sft}=\Prob_{z,\sft}(\cdot|\calZ_n(\sfA)>0)
\end{equation*}
for $z\in\N$ and $\sft\in\{\sfA,\sfB\}$. Convergence in probability with respect to $\Prob^{*}$ is shortly expressed as $\Pstlim$.


\vspace{.2cm}
Theorem \ref{Satz.Proportion.A1} deals with the situation when $\sfB$-parasites multiply at a high rate, viz.\ 
$$ \mu_{0,\sfB}\mu_{1,\sfB}>1, $$
In essence, it asserts that among all contaminated cells in generation $n$ those of type $\sfB$ prevail as $n\to\infty$. This may be surprising at first glance because multiplication of $\sfA$-parasites may also be high (or even higher), namely if
\begin{equation}\tag{SupC}\label{Eq.Supercritical}
 	\mu_{0,\sfA}(\sfAA)^{p_{\sfAA}}\mu_{1,\sfA}(\sfAA)^{p_{\sfAA}}\mu_{0,\sfA}(\sfAB)^{p_{\sfAB}}>1,
\end{equation}
i.e., if the BPRE $(Z_n(\sfA))_{n\ge 0}$ is supercritical. On the other hand, it should be recalled that the subpopulation of $\sfA$-cells grows at rate $\nu<2$ only, whereas the growth rate of $\sfB$-cells is 2. Hence, prevalence of $\sfB$-cells in the subpopulation of all contaminated cells is observed whenever $\#\G^*_n(\sfB)/\#\G_{n}(\sfB)$, the relative proportion of contaminated cells within the $n^{th}$ generation of all $\sfB$-cells, is asymptotically positive as $n\to\infty$.


\begin{Theorem}\label{Satz.Proportion.A1}
 	Assuming $\mu_{0,\sfB}\mu_{1,\sfB}>1$, the following assertions hold true:
	\begin{enumerate}[(a)]
		\item\label{Item.Satz.A1.G}
			\begin{equation*}
 				\frac{\#\G^*_n(\sfA)}{\#\G^*_n}\rightarrow0\quad  \Prob^*\text{-a.s.}
			\end{equation*}
 		\item\label{Item.Satz.A1.FA} Conditioned upon survival of $\sfA$-cells, $F_k(n, 
			\sfB)$ converges to $0$ in probability for any $k\in\N$, i.e.
			\begin{equation*}
	 			\Pstlim_{n\to\infty} F_k(n,\sfB) = 0.
			\end{equation*}
	\end{enumerate}
\end{Theorem}

\vspace{.2cm}
Properties attributed to a high multiplication rate of $\sfA$-parasites are given in Theorem \ref{Satz.Proportion.B1}. First of all, contaminated $\sfB$-cells still prevail in the long-run because, roughly speaking, highly infected $\sfA$-cells eventually produce highly infected $\sfB$-cells whose offspring $m$ generations later for any fixed $m$ are all contaminated (thus $2^{m}$ in number). Furthermore, the $F_{k}(n,\sfA)$ behave as described in \cite{Bansaye:08} for the single-type case: as $n\to\infty$, the number of parasites in any contaminated $\sfA$-cell in generation $n$ tends to infinity and $F_k(n,\sfA)$ to $0$ in probability. Finally, if we additionally assume that type-$\sfB$ parasites multiply faster than type-$\sfA$ parasites, i.e.
$$\mu_{\sfB}:=\mu_{0,\sfB}+\mu_{1,\sfB}>\gamma,$$
then type-$\sfB$ parasites become predominant and $F_k(n,\sfB)$ behaves again in  Bansaye's single-type model \cite{Bansaye:08}.

\begin{Theorem}\label{Satz.Proportion.B1}
	Assuming \eqref{Eq.Supercritical}, the following assertions hold true:
	\begin{enumerate}[(a)]
		\item\label{Item.Satz.B1.FB} Conditioned upon survival of $\sfA$-cells, $F_k(n, 
			\sfA)$ converges to $0$ in probability for any $k\in\N$, i.e.
			\begin{equation*}
	 			\Pstlim_{n\to\infty}F_k(n,\sfA)=0.
			\end{equation*}
		\item\label{Item.Satz.B1.G} \begin{equation*}
			 	\Pstlim_{n\rightarrow\infty}\frac{\#\G^*_n(\sfA)}{\#\G^*_n}=0.
			\end{equation*}
		\item If $\E_{1,\sfB}\calZ_1^2<\infty$, $\mu_{\sfB}>\gamma$ and $					\mu_{0,\sfB}\log\mu_{0,\sfB}+\mu_{1,\sfB}\log\mu_{1,\sfB}<0$, 
		then
			\begin{equation*}
 				\Pstlim_{n\rightarrow\infty}F_k(n,\sfB)=\Prob(\calY(\sfB)=k)
			\end{equation*}
			for all $k\in\N$, where $\Prob(\calY(\sfB)=k)=\lim_{n\rightarrow\infty}
			\Prob_{1,\sfB}(Z_{[n]}=k|Z_{[n]}>0)$.
	\end{enumerate}
\end{Theorem}

\section{Proofs}\label{Section.Proofs}

\begin{Beweis}[Theorem \ref{Satz.Proportion.A1}]
\eqref{Item.Satz.A1.G}
By Theorem \ref{Satz.Erholung}, $2^{-n}\#\G^*_n\to L$ $\Prob^{*}$-a.s. and $\Prob^{*}(L>0)=1$, while Theorem \ref{Satz.ErholungB} shows that $\nu^{-n}\#\G^*_n(\sfA)\to L(\sfA)$ $\Prob$-a.s.\ for an a.s. finite random variable $L(\sfA)$. Consequently,
\begin{equation*}
	\frac{\#\G^*_n(\sfA)}{\#\G^*_n} ~=~ \left(\frac{\nu}{2}\right)^{n}\left(\frac{2^n}
	{\#\G^*_n}\right)\left(\frac
	{\#\G^*_n(\sfA)}{\nu^n}\right) ~\simeq~ \frac{1}{L}\left(\frac{\nu}{2}\right)^{n}
	\frac{\#\G^*_n(\sfA)}{\nu^n}\to 0\quad \Prob^{*}\text{-a.s.} 
\end{equation*}
as $n\to\infty$, for $\nu<2$.

\vspace{.2cm}
\eqref{Item.Satz.A1.FA}
Fix arbitrary $\eps, \delta>0$ and $K\in\N$ and define
	\begin{equation*}
 		D_n := \left \{ \sum_{k=1}^{K}F_k(n,\sfB)>\delta \right \}\cap \Surv(\sfA).
	\end{equation*}
By another appeal to Theorem \ref{Satz.Erholung}, $\#\G^*_n(\sfB)\geq 2^nL$ $\Prob^*$-a.s.\ for all $n\in\N$ and $L$ as above. It follows that
	\begin{align*}
	 	\#\{v\in\G_n(\sfB) : 0< Z_v\leq K\} ~\geq~ \delta\,\#\G^*_n(\sfB)\1_{D_n}
		~\geq~ \delta\,2^n\,L\1_{D_n},
	\end{align*}
and by taking the expectation, we obtain for $m\leq n$
	\begin{align*}
	 	\delta\, &\E\left( L\1_{D_n}\right)
		~\leq~ \frac{1}{2^{n}}\,\E\left(\sum_{v\in\G_n}\1_{\{0<Z_v\leq K, T_v=\sfB
		\}}\right)\\
		&~\leq~ \frac{1}{2^{n}}\,\E\left(\sum_{v\in\G_n}\1_{\{0<Z_v\leq K, T_{v|m}
		=\sfB\}}+\,\#\big\{v\in\G_n:T_{v|m}=\sfA, T_v=\sfB\big\}\right)\\
		&~\leq~ \frac{1}{2^{n}}\,\sum_{v\in\G_n}\Prob\left(0<Z_v\leq K, T_{v|m}=
		\sfB\right)+\frac{1}{2^{m}}\,\E\#\G_{m}(\sfA)\\
		&~\leq~ \frac{1}{2^{n}}\,\sum_{z\ge 1}\sum_{v\in\G_n}\Prob\left(0<Z_v\leq 
		K, Z_{v|m}=z,  T_{v|m}=\sfB\right)+\left(\frac{\nu}{2}\right)^m\\
		&~\leq~ \sum_{z=1}^{\infty}\left(\sum_{u\in\G_m}\frac{\Prob(Z_u=z,
		T_u=\sfB)}{2^{m}}\right)\Bigg(\sum_{u\in\G_{n-m}}\frac{\Prob_{z,\sfB}
		(0<Z_v\leq K)}{2^{n-m}}\Bigg)+\left(\frac{\nu}{2}\right)^m\\
		&~\leq~ \sum_{z=1}^{\infty}\Prob(Z_{[m]}=z, T_{[m]}=\sfB)\,\Prob_{z,\sfB}
		\left(0<Z_{[n-m]}\leq K\right)+\left(\frac{\nu}{2}\right)^m.
	\end{align*}
Since $\nu<2$ we can fix $m\in\N$ such that $(\nu/2)^{m}\le\eps$.
Also fix $z_0\in\N$ such that
	\begin{equation*}
	 \Prob(Z_{[m]}>z_0)\leq\eps.
	\end{equation*}
Then
	\begin{align*}
	  \delta\,\E\left( L\1_{D_n}\right) ~&\leq~ \sum_{z\ge 1}\Prob(Z_{[m]}=z, T_
	  {[m]}=\sfB)\,\Prob_{z,\sfB}\left(0<Z_{[n-m]}\leq K\right)+\left(\frac{\nu}
	  {2}\right)^m\\
	  &\leq~\sum_{z=1}^{z_0}\Prob_{z,\sfB}\left(0<Z_{[n-m]}\leq K\right)
	  +2\eps.
	\end{align*}
But the last sum converges to zero as $n\to\infty$ because, under $\Prob_{z,\sfB}$,
$(Z_{[n]})_{n\ge 0}$ is a single-type BPRE (see \cite{Bansaye:08}) and thus satisfies the extinction-explosion principle. So we have shown that $\E L\1_{D_n}\to 0$
implying $\Prob(D_{n})\to 0$ because $L>0$ on $\Surv$. This completes the proof of the theorem.
\end{Beweis}

Turning to the proof of Theorem \ref{Satz.Proportion.B1}, we first note that part (a) can be directly inferred from Theorem 5.1 in \cite{Bansaye:08} after some minor modifications owing to the fact that $\sfA$-cells do not form a binary tree here but rather a Galton-Watson subtree of it. Thus left with the proof of parts (b) and (c), we first give an auxiliary lemma after the following notation: 
For $v\in\G_n$ and $k\in\N$, let
\begin{equation*}
 \G^*_k(\sft,v):=\{u\in\mathbb G^*_{n+k}(\sft):v<u\}
\end{equation*}
denote the set of all infected $\sft$-cells in generation $n+k$ stemming from $v$.
Let further be
\begin{equation*}
 \mathbb G^*_n(\sfA,\sfB) := \{u\in\G^*_{n+1}(\sfB):T_{u|n}=\sfA\}
\end{equation*}
which is the set of all infected $\sfB$-cells in generation $n+1$ with mother cells of type $\sfA$. 

\begin{Lemma}\label{Lemma.Proportion.B1}
  If \eqref{Eq.Supercritical} holds true, then
	\begin{equation*}
	 	\Pstlim_{n\to\infty}\frac{\#\mathbb G^*_n(\sfA,\sfB)}{\#\mathbb G^*_n(\sfA)} ~=~ \beta ~>~0,
	\end{equation*}
   where $\beta := \lim_{z\to\infty}\E_{z,\sfA}\#\G^{*}_{1}(\sfB)$.
\end{Lemma}

\begin{proof}
Since $z\mapsto\E_{z,\sfA}\#\G^{*}_{1}(\sfB)$ is increasing and $\E_{1,\sfA}\#\G^{*}_{1}(\sfB)>0$ by our standing assumption \eqref{SA5}, we see that $\beta$ must be positive. Next observe that, for each $n\in\N$,
  \begin{equation*}
    \#\G^*_n(\sfA,\sfB) = \sum_{v\in\G^*_{n-1}(\sfA)}\#\G^*_1(\sfB,v),
  \end{equation*}
where the $\#\G^*_1(\sfB,v)$ are conditionally independent given $\calZ_n(\sfA)>0$. Since $\#\G^*_n(\sfA)\to\infty$ $\Prob^*$-a.s. (Theorem \ref{Satz.ExplosionInfZellen}) and $\Prob^n=\Prob(\cdot|\calZ_n(\sfA)>0)\xrightarrow{w} \Prob^*$, it is not difficult to infer with the help of the SLLN that
  \begin{equation*}
    \frac{\#\mathbb G^*_n(\sfA,\sfB)}{\#\mathbb G^*_n(\sfA)} ~\overset{\Prob^*}\simeq~         
    \frac{1}{\#\G^*_n(\sfA)}\sum_{v\in\G^*_n(\sfA)}\E_{Z_v,\sfA}\#\G^*_1(\sfB),\quad n
    \to\infty.
  \end{equation*}
where $a_{n}\stackrel{\Prob^{*}}{\simeq}b_{n}$ means that $\Prob^{*}(a_{n}/b_{n}\to 1)=1$. Now use $\E_{z,\sfA}\#\G^*_1(\sfB)\uparrow\beta$ to infer the existence of a $z_0\in\N$ such that for all $z\geq z_0$
  \begin{equation*}
    \E_{z,\sfA}\#\G^*_1(\sfB) \geq \beta(1-\eps).
  \end{equation*}
After these observations we finally obtain by an appeal to Theorem \ref{Satz.Proportion.B1}(a) that
  \begin{align*}
   \beta~&\geq~\frac{1}{\#\G^*_n(\sfA)}\sum_{v\in\G^*_n(\sfA)}\E_{Z_v,\sfA}\#\G^*_1(\sfB)\\ 
  &\geq~\sum_{z\geq z_0}\frac{F_z(n,\sfA)}{\#\{v\in\G^*_n(\sfA)|Z_v\geq z_0\}}\sum_{v\in\{u\in\G^*_n(\sfA)|Z_u\geq z_0\}}\E_{Z_v,\sfA}\#\G^*_1(\sfB)\\
  &\geq~\beta(1-\eps)\sum_{z\geq z_0}F_z(n,\sfA)\\
  &\to~ \beta(1-\eps),\quad n\to\infty.
  \end{align*}
This completes the proof of the lemma.
\end{proof}

\begin{Beweis}[Theorem \ref{Satz.Proportion.B1}(b) and (c)]
Let $\eps>0$ and $N\in\N$. Then
	\begin{align*}
	  \#\G^*_n(\sfB) ~&=~ \sum_{k=0}^{n-1}\sum_{v\in\G^*_k(\sfA,\sfB)}\#\mathbb 
	  G^*_{n-k-1}(\sfB,v)\\
	  &\geq~\sum_{k=0}^{n-1}\sum_{v\in\{u\in\mathbb G^*_k(\sfA,\sfB) | Z_u\geq z
	  \}}\#\mathbb G^*_{n-k-1}(\sfB,v)\\
	  &\geq~\sum_{v\in\{u\in\mathbb G^*_{n-1-m}(\sfA,\sfB) | Z_u\geq z\}}\#\mathbb 
	  G^*_{m}(\sfB,v)
	\end{align*}
a.s. for all $n>m\ge 1$ and $z\in\N$, and thus
	\begin{equation}\label{Eq.GA.nur ein m}
	  \begin{split}
	    \Prob^*&\left(\frac{\#\mathbb G^*_n(\sfA)}{\#\mathbb G^*_n}>\frac{1}{N
	    +1}\right)~=~\Prob^*\left(N\,\#\mathbb G^*_n(\sfA) > 
	    \#\mathbb G^*_n(\sfB)\right)\\
	    &\leq~\Prob^*\left(N\,\#\mathbb G^*_n(\sfA)>\sum_{v\in\{u\in\mathbb G^*
	    _{n-1-m}(\sfA,\sfB) | Z_u\geq z\}}\#\mathbb G^*_{m}(\sfB,v)\right).
	  \end{split}
	\end{equation}
Fix $m$ so large that
	\begin{equation*}
	  2^{m}(1-\eps)>\frac{4N}{\beta}.
	\end{equation*}
Then, since 
	\begin{equation*}
	 	\lim_{z\rightarrow\infty}\Prob_{z,\sfB}(\#\mathbb G^*_{m}=2^{m})=1,
	\end{equation*}
there exists $z_0\in\N$ such that
	\begin{equation*}
	 	\Prob_{z,\sfB}(\#\mathbb G^*_{m}=2^{m})~\geq~1-\eps
	\end{equation*}
and therefore
	\begin{equation}\label{Eq.Th.B1.1}
	  \E_{z,\sfB}\#\mathbb G^*_{m}~\geq~(1-\eps)2^{m} ~>~\frac{4N}{\beta}
	\end{equation}
for all $z\geq z_0$. Moreover, $\sum_{k\geq z_0}F_k(n,\sfA)\xrightarrow{\Prob^*} 1$
by part (a), whence
	\begin{equation*}
	  \frac{\#\{v\in\G^*_n(\sfA,\sfB):Z_v\geq z_0\}}{\#\G^*_n(\sfA,\sfB)}\xrightarrow
	  {\Prob^*}1.
	\end{equation*}
This together with Lemma \ref{Lemma.Proportion.B1} yields
	\begin{equation*}
	  \frac{\#\{v\in\G^*_n(\sfA,\sfB):Z_v\geq z_0\}}{\#\G^*_n(\sfA)}\xrightarrow{\Prob^*}\beta
	\end{equation*}
and thereupon
	\begin{equation}\label{Eq.Th.B1.2}
	  \Prob^*\left(\frac{\#\{v\in\G^*_n(\sfA,\sfB):Z_v\geq z_0\}}{\#\G^*_n(\sfA)}\geq 
	  \frac{\beta}{2}\right)\geq1-\eps
	\end{equation}
for all $n\geq n_0$ and some $n_{0}\in\N$. By combining \eqref{Eq.GA.nur ein m} and \eqref{Eq.Th.B1.2}, we now infer for all $n\geq n_0+m$
	\begin{align*}
	  &\Prob^*\left(\frac{\#\mathbb G^*_n(\sfA)}{\#\mathbb G^*_n}>\frac{1}{N
	  +1}\right)\\
	  &\leq~\Prob^*\left(N\,\#\mathbb G^*_n(\sfA) > \sum_{v\in\{u\in\mathbb G^*_{n-1-
	  m}(\sfA,\sfB) : Z_u\geq z\}}\#\mathbb G^*_{m}(\sfB,v)\right)\\
	  &\leq~\Prob^*\left(\frac{2N}{\beta} > \frac{\sum_{v\in\{u\in\mathbb 
	  G^*_{n-1-m}(\sfA,\sfB) : Z_u\geq z\}}\#\mathbb G^*_{m}(\sfB,v)}
	  {\displaystyle\#\{u\in\mathbb G^*_{n-1-m}(\sfA,\sfB):Z_u\geq z\}}\right)+\eps\\
	  &\leq~\Prob^{\,n-m}\left(\frac{2N}{\beta} > \frac{\sum_{v\in\{u\in
	  \mathbb G^*_{n-1-m}(\sfA,\sfB):Z_u\geq z\}}\#\mathbb G^*_{m}(\sfB,v)}
	  {\displaystyle\#\{u\in\mathbb G^*_{n-1-m}(\sfA,\sfB):Z_u\geq z\}}\right)\frac
	  {\Prob(\calZ_{n-m}(\sfA)>0)}{\Prob(\Surv(\sfA))}+\eps\\
	  &\leq~\Prob^{\,n-m}\left(\frac{2N}{\beta} > \frac{\sum_{i=1}^{\#\{u
	  \in\mathbb G^*_{n-1-m}(\sfA,\sfB):Z_u\geq z\}}\mathcal G_{i,m}(z_0)}
	  {\displaystyle\#\{u\in\mathbb G^*_{n-1-m}(\sfA,\sfB):Z_u\geq z\}}\right)\frac
	  {\Prob(\calZ_{n-m}(\sfA)>0)}{\Prob(\Surv(\sfA))}+\eps\\
	\end{align*}
where the $\mathcal G_{i,m}(z_0)$ are iid with the same law as $\#\{v\in\G^*_{m}(\sfB):Z_{\varnothing}=z_0, T_{\varnothing}=\sfB\}$. The LLN provides us with $n_1\geq n_0+m$ such that for all $n\geq n_1$
	\begin{equation*}
	  \Prob^{n-m}\left(\frac{\sum_{i=1}^{\#\{u\in\mathbb G^*_{n-1-m}(\sfA,\sfB): 
	  Z_u\geq z\}}\mathcal G_{i,m}(z_0)}{\#\{u\in\mathbb G^*_{n-1-m}(\sfA,
	  \sfB) :Z_u\geq z\}}\geq  \E\mathcal G_{i,m}(z_0)/2\right)\geq 1-\eps.
	\end{equation*}
By combining this with \eqref{Eq.Th.B1.1}, we can further estimate in the above inequality
	\begin{align*}
	  \Prob^*&\left(\frac{\#\mathbb G^*_n(\sfA)}{\#\mathbb G^*_n}>\frac{1}{N
	  +1}\right)\\
	  &\leq~\left(\Prob^{n-m}\left(\frac{2N}{\beta} > \E\mathcal G_{i,m}(z_0)/2 > \frac{2N}{\beta}\right) +\eps\right)
	  \frac{\Prob(\calZ_{n-m}(\sfA)>0)}{\Prob(\Surv(\sfA))}+\eps\\
	  &=~\left(\frac{\Prob(\calZ_{n-m}(\sfA)>0)}{\Prob(\Surv(\sfA))}+1\right)\eps~\xrightarrow{n\to\infty}~2\eps.
	\end{align*}
This completes the proof of part (b).

\vspace{.2cm}
As for (c), we will show that all conditions needed by Bansaye \cite{Bansaye:08} to prove his Theorem 5.2 are fulfilled. Our assertions then follow along the same arguments as provided there.

\vspace{.2cm}
\textsc{Step 1:} $(\mu_{\sfB}^{-n}\calZ_n(\sfB))_{n\ge 0}$ is a submartingale and converges a.s.\  to a finite limit $W(\sfB)$. The submartingale property follows from 
\begin{align*}
  &\E(\calZ_{n+1}(\sfB)|\calZ_n(\sfB))
  ~=~\E\left(\sum_{v\in\G^*_n}\big(Z_{v0}\1_{\{T_{v0}=\sfB\}}+Z_{v1}\1_{\{T_{v1}=\sfB\}}\big)\Bigg|\calZ_n(\sfB)\right)\\
  &=~\calZ_n(\sfB)\E\left(X^{(0)}(\sfB)+X^{(1)}(\sfB)\right)+\E\left(\sum_{v\in\G^*_n(\sfA)}\hspace{-.2cm}\big(Z_{v0}\1_{\{T_{v0}=\sfB\}}+Z_{v1}\1_{\{T_{v1}=\sfB\}}\big)\Bigg|\calZ_n(\sfB)\right)\\
  &\geq~\calZ_n(\sfB)\mu_{\sfB}
\end{align*}
for any $n\in\N$, while the a.s.\ convergence is a consequence of
\begin{equation*}
 \sup_{n\in\N}\E\left(\frac{\calZ_{n}(\sfB)}{\mu_{\sfB}^{n}}\right)~<~\infty
\end{equation*}
which, using our assumption $\gamma<\mu_{\sfB}$, follows from
\begin{align*}
  \E\left(\frac{\calZ_{n+1}(\sfB)}{\mu_{\sfB}^{n+1}}\right) ~&=~ \E\left(\frac{\calZ_{n}(\sfB)}{\mu_{\sfB}^{n}}\right)+\E\left(\frac{1}{\mu_{\sfB}^{n+1}}\sum_{v\in\G^*_n(\sfA)}Z_{v0}\1_{\{T_{v0}=\sfB\}}+Z_{v1}\1_{\{T_{v1}=\sfB\}}\right)\\
  &=~\E\left(\frac{\calZ_{n}(\sfB)}{\mu_{\sfB}^{n}}\right)+\frac{1}{\mu_{\sfB}^{n+1}}\E\left(\calZ_n(\sfA)\underbrace{\E(Z_{0}\1_{\{T_{0}=\sfB\}}+Z_{1}\1_{\{T_{1}=\sfB\}})}_{=:\mu_{\sfAB}}\right)\\
  &=~\E\left(\frac{\calZ_{n}(\sfB)}{\mu_{\sfB}^{n}}\right)+\frac{\mu_{\sfAB}}{\mu_{\sfB}}\left(\frac{\gamma}{\mu_{\sfB}}\right)^n\\
   =...&=~\frac{\mu_{\sfAB}}{\mu_{\sfB}}\sum_{k=0}^n\left(\frac{\gamma}{\mu_{\sfB}}\right)^k\
 \ \leq~\frac{\mu_{\sfAB}}{\mu_{\sfB}}\sum_{k=0}^{\infty}\left(\frac{\gamma}{\mu_{\sfB}}\right)^k~<~\infty
\end{align*}
for any $n\in\N$.

\vspace{.2cm}
\textsc{Step 2:} $\{W(\sfB)=0\}=\Ext$ a.s.\\
The inclusion $\supseteq$ being trivial, we must only show that $\Prob(W(\sfB)>0)\geq\Prob(\Surv)$. For $i\ge 1$, let $(\calZ_{i,n}(\sfB))_{n\ge 0}$ be iid copies of $(\calZ_{n}(\sfB))_{n\ge 0}$ under $\Prob_{1,\sfB}$. Each $(\calZ_{i,n}(\sfB))_{n\ge 0}$ forms a Galton-Watson process which dies out iff $\mu_{\sfB}^{-n}\calZ_{i,n}(\sfB)\to 0$ (see \cite{Bansaye:08}). Then for all $m,N\in\N$, we obtain
\begin{align*}
  \Prob(W(\sfB)>0) ~&=~\Prob\left(\lim_{n\to\infty}\frac{\calZ_{m+n}(\sfB)}{\mu_{\sfB}^{m+n}}>0\right)\\
  &\geq~\Prob\left(\lim_{n\to\infty}\frac{1}{\mu_{\sfB}^m}\sum_{i=1}^{\calZ_{m}(\sfB)}\frac{\calZ_{i,n}(\sfB)}{\mu_{\sfB}^{n}}>0\right)\\
  &\geq~\Prob\left(\lim_{n\to\infty}\frac{1}{\mu_{\sfB}^{m}}\sum_{i=1}^{\calZ_{m}(\sfB)}\frac{\calZ_{i,n}(\sfB)}{\mu_{\sfB}^{n}}>0, \calZ_m(\sfB)\geq N\right)\\
  &\geq~\Prob\left(\lim_{n\to\infty}\sum_{i=1}^{N}\frac{\calZ_{i,n}(\sfB)}{\mu_{\sfB}^{n}}>0, \calZ_{m}(\sfB)\geq N\right)\\
  &\geq~\Prob\left(\calZ_m(\sfB)\geq N\right)-\Prob_{1,\sfB}\left(\lim_{n\to\infty}\sum_{i=1}^{N}\frac{\calZ_{i,n}(\sfB)}{\mu_{\sfB}^{n}}=0\right)\\
  &=~\Prob\left(\calZ_{m}(\sfB)\geq N\right)-\Prob_{1,\sfB}\left(\lim_{n\to\infty}\frac{\calZ_{n}(\sfB)}{\mu_{\sfB}^{n}}=0\right)^N\\
  &=~\Prob\left(\calZ_{m}(\sfB)\geq N|\Surv\right)\,\Prob(\Surv)-\Prob_{1,\sfB}\left(\Ext\right)^N.
\end{align*}
and then, upon letting $m$ and $N$ tend to infinity,
\begin{equation*}
 \Prob(W(\sfB)>0)~\geq~\Prob(\Surv)
\end{equation*}
because $\Prob_{1,\sfB}\left(\Ext\right)<1$ and by Theorem \ref{Satz.ExplosionInfZellen}.

\vspace{.2cm}
\textsc{Step 3:} $\sup_{n\ge 0}\E\xi_{n}<\infty$, where $\xi_{n}:=\left(\mu_{\sfB}/2\right)^{-n} Z_{[n]}$.

\vspace{.1cm}\noindent
First, we note that $(Z_{[n]})_{n\ge 0}$, when starting with a $\sfB$-cell hosting one parasite (under $\Prob_{1,\sfB}$), is a BPRE with mean $\mu_{\sfB}/2$ (see \cite{Bansaye:08}). Second, we have
\begin{equation*}
 \E Z_{[n]}\1_{\{T_{[n]}=\sfA\}} ~=~ \Prob(T_{[n]}=\sfA)\,\E Z_{n}(\sfA) ~=~ \left(\frac{\gamma}{2}\right)^n
\end{equation*}
and thus
\begin{align*}
  \E Z_{[n]} ~&=~ \E Z_{[n]}\1_{\{T_{[n]}=\sfA\}}+\sum_{m=0}^{n-1}\E Z_{[n]}\1_{\{T_{[m]}=\sfA,T_{[m+1]}=\sfB\}}\\
  &=~\left(\frac{\gamma}{2}\right)^n+\sum_{m=0}^{n-1}\E Z_{[m]}\1_{\{T_{[m]}=\sfA\}}\,\E_{1,\sfA}Z_{[1]}\1_{\{T_{[1]}=\sfB\}}\,\E_{1,\sfB}Z_{[n-m-1]}\\
  &=~\left(\frac{\gamma}{2}\right)^n+\eta\sum_{m=0}^{n-1}\left(\frac{\gamma}{2}\right)^m\left(\frac{\mu_{\sfB}}{2}\right)^{n-m-1}
\end{align*}
for all $n\in\N$, where $\eta:=\E_{1,\sfA}Z_{[1]}\1_{\{T_{[1]}=\sfB\}}$. This implies
\begin{equation}\label{Eq.BPRE.l1beschr}
  \sup_{n\in\N}\E\xi_n ~=~ \left(\frac{\gamma}{\mu_{\sfB}}\right)^n+\frac{2\eta}{\mu_{\sfB}}\sum_{m=0}^{n-1}\left(\frac{\gamma}{\mu_{\sfB}}\right)^m~\leq~c\sum_{m=0}^{\infty}\left(\frac{\gamma}{\mu_{\sfB}}\right)^m~<\infty
\end{equation}
for some $c<\infty$.

\vspace{.2cm}
\textsc{Step 4:} $\lim_{K\to\infty}\sup_{n\ge 0}\E\xi_n\1_{\{Z_{[n]}\geq K\}}=0$.

\vspace{.1cm}\noindent
By our assumptions, $(Z_{[n]})_{n\ge 0}$, when starting in a $\sfB$-cell with one parasite, is a strongly subcritical BPRE with mean $\mu_{\sfB}/2$ (see \cite{Bansaye:08}). Hence, by \cite[Corollary 2.3]{Afanasyev+etal:05}, 
\begin{equation}\label{Eq.B1c.gi}
  \lim_{K\to\infty}\sup_{n\ge 0}\E_{1,\sfB}\xi_n\1_{\{Z_{[n]}> K\}} ~=~0,
\end{equation}
which together with \eqref{Eq.BPRE.l1beschr} implies for $n,m\in\N$
\begin{align*}
 \lim_{K\to\infty}\sup_{n\ge 0}\E\, &\xi_{n+m}\1_{\{Z_{[n+m]}>K\}}\\
 ~&\leq~ \lim_{K\to\infty}\sup_{n\ge 0}\E\,\xi_{n+m}\1_{\{Z_{[n+m]}>K\}}\1_{\{T_{[m]}=\sfB\}}+\sup_{n\ge 0}\E\,\xi_{n+m}\1_{\{T_{[m]}=\sfA\}}\\
  &\leq~ \lim_{K\to\infty}\sup_{n\ge 0}\E_{1,\sfB}\xi_n\1_{\{Z_{[n]}>K\}}\,\E \xi_m+\E \xi_m\1_{\{T_{[m]}=\sfA\}}\,\sup_{n\in\N}\E\,\xi_n\\
  &\leq~ \left(\frac{\gamma}{\mu_{\sfB}}\right)^m\sup_{n\in\N}\E\,\xi_n
\end{align*}
and the last expression can be made arbitrarily small by choosing $m$ sufficiently large, for $\gamma<\mu_{\sfB}$. This proves Step 4.

\vspace{.2cm}
\textsc{Final step:}
Having verified all conditions needed for the proof of Theorem 5.2 in \cite{Bansaye:08}, one can essentially follow his arguments to prove Theorem \ref{Satz.Proportion.B1}(c). We refrain from supplying all details here and restrict ourselves to an outline of the main ideas. First use
what has been shown as \textsc{Step 1 - 4} to prove an analogue of \cite[Lemma 6.5]{Bansaye:08}, i.e. \textit{(control of filled-in cells)}
\begin{equation}\label{control.filled.in.cells}
 	\lim_{K\to\infty}\sup_{n,q\geq0}\Prob^*\left(\frac{\#\{v\in\G^*_{n+q}(\sfB):Z_{v|n}>K\}}{\#\G^*_{n+q}(\sfB)}\geq\eta\right)=0
\end{equation}
for all $\eta>0$, and of \cite[Prop.\ 6.4]{Bansaye:08}, i.e. \textit{(separation of descendants of parasites)}
\begin{equation}\label{sep.of.descendants}
 \lim_{q\to\infty}\sup_{n\geq0}\Prob^*\left(\frac{\#\{v\in\G^*_{n+q}(\sfB):Z_{v|n}\leq K, N_n(v)\geq2\}}{\#\G^*_{n+q}(\sfB)}\geq\eta\right)=0
\end{equation}
for all $\eta>0, K\in\N$, where $N_n(v)$ denotes the number of parasites in cell $v|n$ with at least one descendant in cell $v$. In particular, \eqref{control.filled.in.cells} (with $q=0$) combined with $\#\G^*_n(\sfB)\to\infty$ $\Prob^*$-a.s. implies the existence of a $K_0\geq0$ such that for all $N\in\N$
\begin{equation}\label{concentration.of.parasites}
 \lim_{n\to\infty}\inf_{K\geq K_0}\Prob^*\left(\sum_{v\in\G^*_n(\sfB)}Z_v\1_{\{Z_v\leq K\}}\geq N\right) = 1.
\end{equation}
Using \eqref{control.filled.in.cells} and \eqref{sep.of.descendants}, we infer that, for all $\eta, \eps>0$, there exist $K_1\geq K_0$ and $q_0\in\N$ such that for all $n\in\N$
\begin{equation*}
 \Prob^*\Bigg(\Bigg|F_k(n+q_0,\sfB)-\underbrace{\frac{\#\{v\in\G^*_{n+q_0}(\sfB) | Z_v=k, Z_{v|n}\leq K_1, N_n(v)=1\}}{\#\{v\in\G^*_{n+q_0}(\sfB) | Z_{v|n}\leq K_1, N_n(v)=1\}}}_{=: J_{n}}\Bigg|\geq\eta\Bigg)~\leq~\eps.
\end{equation*}
Since $\#\G^*_n(\sfA)/\#\G^*_n(\sfB)\xrightarrow{\Prob^*}0$, we further get
\begin{align*}
 J_{n} ~&\overset{\Prob^*}{{\underset{n\to\infty}{\simeq}}}~ \frac{\#\{v\in\G^*_{n+q_0}(\sfB) | Z_v=k, Z_{v|n}\leq K_1, T_{v|n}=\sfB, N_n(v)=1\}}{\#\{v\in\G^*_{n+q_0}(\sfB) | Z_{v|n}\leq K_1, T_{v|n}=\sfB, N_n(v)=1\}}
\end{align*}
as $n\to\infty$, which puts us in the same situation as in the proof of \cite[Thm.\ 5.2]{Bansaye:08}. Now, by using \eqref{concentration.of.parasites} and the LLN, we can identify the limit of $J_n$ which is in fact the same as in Step 1 of the proof of \cite[Thm.\ 5.2]{Bansaye:08}. A reproduction of the subsequent arguments from there finally establishes the result.
\end{Beweis}

\section*{Acknowledgment}

We are indebted to Joachim Kurtz (Institute of Evolutionary Biology, University of M\"unster) 
for sharing with us his biological expertise of host-parasite coevolution and many fruitful discussions that helped us to develop the model studied in this paper.

\section{Glossary}

\begin{tabular}[c]{p{0.25\textwidth} p{0.75\textwidth}}
 	$\T$		& cell tree\\
	$\G_n$		& set of cells in generation $n$\\
	$\G_n(\sft)$	& set of cells of type $\sft$ in gegeration $n$\\
	$\G^*_n$	& set of contaminated cells in generation $n$\\
	$\G^*_n(\sft)$	& set of contaminated cells of type $\sft$ in generation $n$\\
	$T_v$		& type of cell $v$\\
	$p_{\sfs}$		& probability that the daughter cell of an $\sfA$-cell is of type $\sfs$\\
	$p_0$		& probability that the $1^{st}$ daughter cell of an $\sfA$-cell 
				is of type $\sfA$\\
	$p_1$		& probability that the $2^{nd}$ daughter cell of an $\sfA$-cell 
				is of type $\sfA$\\
	$\nu$		& mean number of type-$\sfA$ daughter cells of an $\sfA$-cell\\
	$(X^{(0)}(\sfA,\sfs),X^{(1)}(\sfA,\sfs))$ & offspring numbers of an $\sfA$-parasite with 
				daughter cells of type $\sfs\in\{\sfAA,\sfAB,\sfBB\}$\\
	$(X^{(0)}(\sfB),X^{(1)}(\sfB))$	& offspring numbers of a $\sfB$-parasite\\
	$Z_v$		& number of parasites in cell $v$\\
	$\mu_{i,\sft}(\sfs)$	& mean number of offspring of a $\sft$-parasite which goes in 
				daughter cell $i\in\{0,1\}$ if daughter cells are of type 
				$\sfs\in\{\sfAA,\sfAB,\sfBB\}$\\
	$\mu_{i,\sfB}$	& mean offspring number of $\sfB$-parasites which go in 
				daughter cell $i\in\{0,1\}$\\
	$\mu_{\sfB}$	& reproduction mean of a parasite in a $\sfB$-cell\\
	$\calZ_n$		& number of parasites in generation $n$\\
	$\Z_n(\sft)$ 	& number of parasites in $\sft$-cells in generation $n$\\
	$\Ext/\Surv$		& event of extinction/survival of parasites\\
	$\Ext(\sft)/\Surv(\sft)$	& event of extinction/survival of $\sft$-parasites\\
	$Z_{[n]}$	& number of parasites in a random cell in generation $n$\\
	$Z_n(\sfA)$	& number of parasites of a random $\sfA$-cell in generation $n$\\
	$f_{n}(s|\Lambda),\ f_{n}(s)$ &quenched and annealed generating function of $Z_{n}(\sfA)$, respectively\\
	$g_{\Lambda_n}(s)$	& generating function giving the $n$-th reproduction law of the 
				process of a random $\sfA$ cell line\\
	$\gamma$	& mean number of offspring of an $\sfA$-parasite which go in an $\sfA$-cell\\
	$\hat\gamma$	& $:=\nu\,\E g^{\prime}_{\Lambda_1}(1)^2
		=p_{\sfAA}\left(\mu^2_{0,\sfA}(\sfAA)+\mu^2_{1,\sfA}(\sfAA)\right)
		+p_{\sfAB}\,\mu^2_{0,\sfA}(\sfAB)$\\
	$\Prob_{z,\sft}$	& probability measure under which the process starts with one $\sft$-cell 
	containing $z$ parasites\\
	$\Prob^*_{z,\sft}$	& the same as before but conditioned upon $\Surv(\sfA)$\\
	$\Prob^n_{z,\sft}$	& the same as before but conditioned upon survival of 
	$\sfA$-parasites in generation $n$
\end{tabular}


\end{document}